\documentclass[12pt,twoside,leqno]{article}
\usepackage[T1]{fontenc}
\usepackage{amsfonts}
\usepackage{amsmath,amsthm}
\usepackage{amssymb,latexsym}
\usepackage{enumerate}

\theoremstyle{plain}
\newtheorem{theorem}{Theorem}[section]

\newtheorem{proposition}[theorem]{Proposition}
\newtheorem{corollary}[theorem]{Corollary}

\theoremstyle{definition}

\newtheorem{definition}[theorem]{Definition}
\theoremstyle{remark}
\newtheorem{remark}[theorem]{Remark}
\newtheorem{question}[theorem]{Question}

\DeclareMathOperator{\Iso}{{\rm Iso}}

\DeclareMathOperator{\cl}{{\rm cl}}
\DeclareMathOperator{\scl}{{\rm scl}}
\DeclareMathOperator{\sqc}{{\rm sc}}

\begin{document}
\title{$k$-spaces, sequential spaces and related topics in the absence of the axiom of choice}
\author{Kyriakos Keremedis and Eliza Wajch\\
Department of Mathematics, University of the Aegean\\
Karlovassi, Samos 83200, Greece\\
kker@aegean.gr\\
Institute of Mathematics\\
Faculty of Exact and Natural Sciences \\
Siedlce University of Natural Sciences and Humanities\\
ul. 3 Maja 54, 08-110 Siedlce, Poland\\
eliza.wajch@wp.pl}
\maketitle
\begin{abstract}
In the absence of the axiom of choice, new results concerning sequential, Fr\'echet-Urysohn, $k$-spaces, very $k$-spaces, Loeb and Cantor completely metrizable spaces are shown. New choice principles are introduced. Among many other theorems, it is proved in $\mathbf{ZF}$ that every Loeb, $T_3$-space having a base expressible as a countable union of finite sets is a metrizable second-countable space whose every $F_{\sigma}$-subspace is separable; moreover, every $G_{\delta}$-subspace of a second-countable, Cantor completely metrizable space is Cantor completely metrizable, Loeb and separable. It is also noticed that Arkhangel'skii's statement that every very $k$-space is Fr\'echet-Urysohn is unprovable in $\mathbf{ZF}$ but it holds in $\mathbf{ZF}$ that every first-countable, regular very $k$-space whose family of all non-empty compact sets has a choice function is Fr\'echet-Urysohn. That every second-countable metrizable space is a very $k$-space is equivalent to the axiom of countable choice for $\mathbb{R}$. 
 \medskip
 

\end{abstract}

\section{Introduction}
\label{s1}
\subsection{Set-theoretic framework}
The intended context for reasoning in this paper is the Zermelo-Fraenkel
set theory $\mathbf{ZF}$ in which the axiom of choice $\mathbf{AC}$ is deleted. To stress that a result is proved in $\mathbf{ZF}$ (respectively, only in $\mathbf{ZF+A}$ or $\mathbf{ZFC}$ where $\mathbf{ZFC}$ is $\mathbf{ZF+AC}$), we write at the beginning of the statements of the theorems $[\mathbf{ZF}]$ (respectively, $[\mathbf{ZF+A}]$ or $[\mathbf{ZFC}]$). Apart from models of $\mathbf{ZF}$, we refer to some models of set theory $\mathbf{ZFA}$ (or $\text{ZF}^0$ in \cite{hr}) with an infinite set of atoms (see \cite{je}).  All our theorems proved in $\mathbf{ZF}$ are also theorems of $\mathbf{ZFA}$. 

The set of all finite ordinal numbers (of von Neumann) is denoted by $\omega$. For convenience, we denote by $\mathbb{N}$ the set $\omega\setminus\{0\}$. 

For a set $X$, the power set of $X$ is denoted by $\mathcal{P}(X)$. For sets $X$ and $Y$, we write $|X|\leq |Y|$ if $X$ is equipotent with a subset of $Y$. 

\begin{definition}
\label{s1d1}
Let $X$ be a set.
\begin{enumerate}
\item[(i)] $X$ is called \emph{finite} if there exists $n\in\omega$ such that $X$ is equipotent with $n$; otherwise, $X$ is called \emph{infinite}.
\item[(ii)] $X$ is called \emph{Dedekind-finite} if no proper subset of $X$ is equipotent with $X$; otherwise, $X$ is called \emph{Dedekind-infinite}.
\item[(iii)] $\mathbf{X}$ is called \emph{countable} if $X$ is equipotent with a subset of $\omega$.
\item[(iv)] $\mathbf{X}$ is called \emph{denumerable} if $X$ is a countable infinite set.
\item[(v)] $\mathbf{X}$ is called a \emph{cuf set} if $X$ is expressible as a countable union of finite sets.
\end{enumerate}
\end{definition}

For a set $X$, the family of all finite subsets of $X$ is denoted by $[X]^{<\omega}$, and the family of all countable subsets of $X$ is denoted by $[X]^{\leq\omega}$. We denote by $\mathcal{P}^{inf}(X)$ the collection of all infinite subsets of $X$, and by $\mathcal{P}^{un}(X)$ the collection of all uncountable subsets of $X$.

Let $J$ be a non-empty set and let  $\mathcal{A}=\{A_j: j\in J\}$ be a family of non-empty sets. Then every function $f\in\prod_{j\in J}A_j$ is called a \emph{choice function} of $\mathcal{A}$. Every function $f\in\prod_{j\in J}\mathcal{P}(A_j)$ such that, for every $j\in J$, $f(j)$ is a non-empty finite subset of $A_j$, is called a \emph{multiple choice function} of $\mathcal{A}$. If $J$ is infinite, then it is said that $\mathcal{A}$ has a partial (multiple) choice function if there exists an infinite subset $S$ of $J$ such that the family $\{A_j: j\in S\}$ has a (multiple) choice function. If $S$ is an infinite subset of $J$ and $f\in\prod_{j\in S}A_j$ (respectively, $f\in\prod_{j\in S}([A_j]^{<\omega}\setminus\{\emptyset\}))$, then $f$ is called a \emph{partial choice function} (respectively, a \emph{partial multiple choice function}) of $\mathcal{A}$.

\subsection{Topological terminology and notation}

Let $\mathbf{X}=\langle X, \tau\rangle$ and let $A\subseteq X$. The closure of $A$ in $\mathbf{X}$ is denoted by $\cl_{\mathbf{X}}(A)$. The collection $\tau|_ A=\{A\cap U: U\in\tau\}$ is the subspace topology on $A$ inherited from the topology of $\mathbf{X}$. If this is not misleading, the (topological) subspace $\mathbf{A}=\langle A, \tau|_A\rangle$ of $\mathbf{X}$ will be also denoted by $A$. If it is not stated otherwise, all subsets of $X$ will be considered as topological subspaces of $\mathbf{X}$. The set $A$ is called \emph{sequentially closed} in $\mathbf{X}$ if, for every $x\in X$ and every sequence of points of $A$ converging in $\mathbf{X}$ to $x$, we have $x\in A$. Let $A^{\ast}$ be the set of all points $x\in X$ such that there exists a sequence of points of $A\setminus\{x\}$ which converges in $\mathbf{X}$ to the point $x$. Then the set $\sqc_{\mathbf{X}}(A)=A\cup A^{\ast}$ is called the \emph{sequential closure} of $A$. The intersection of all sequentially closed subsets of $\mathbf{X}$ containing $A$ is denoted by $\scl_{\mathbf{X}}(A)$. Clearly, $\sqc_{\mathbf{X}}(A)\subseteq \scl_{\mathbf{X}}(A)\subseteq \cl_{\mathbf{X}}(A)$. The set $A$ is sequentially closed in $\mathbf{X}$ if and only if $A=\sqc_{\mathbf{X}}(A)$ or, equivalently, $A=\scl_{\mathbf{X}}(A)$.

\begin{definition} 
\label{s1d2}
If $\mathbf{X}$ is a topological space, then:
\begin{enumerate}
\item[(i)] $\mathcal{C}l(\mathbf{X})$ is the collection of all closed sets of $\mathbf{X}$, and 
$$\mathcal{C}l^{\ast}(\mathbf{X})=\mathcal{C}l(\mathbf{X})\setminus\{\emptyset\};$$

\item[(ii)] $\mathcal{SC}(\mathbf{X})$ is the collection of all sequentially closed sets of $\mathbf{X}$, and 
$$\mathcal{SC}^{\ast}(\mathbf{X})=\mathcal{SC}l(\mathbf{X})\setminus\{\emptyset\};$$

\item[(iii)] $\mathcal{K}(\mathbf{X})$ is the collection of all compact subsets of $\mathbf{X}$, and 
$$\mathcal{K}^{\ast}(\mathbf{X})=\mathcal{K}(\mathbf{X})\setminus\{\emptyset\};$$

\item[(iv)]  $\mathcal{F}_{\sigma}(\mathbf{X})$ is the collection of all $F_{\sigma}$-sets in $\mathbf{X}$, and 
$$\mathcal{F}_{\sigma}^{\ast}(\mathbf{X})=\mathcal{F}_{\sigma}(\mathbf{X})\setminus\{\emptyset\};$$

\item[(v)] $\mathcal{G}_{\delta}(\mathbf{X})$ is the collection of all $G_{\delta}$-sets in $\mathbf{X}$, and 
$$\mathcal{G}_{\delta}^{\ast}(\mathbf{X})=\mathcal{G}_{\delta}(\mathbf{X})\setminus\{\emptyset\}.$$
\end{enumerate}
\end{definition}

\begin{definition} 
\label{s1d3}
Let $\mathbf{X}$ be a topological space.
\begin{enumerate}
\item[(i)] $\mathbf{X}$ is called \emph{Loeb} if $\mathcal{C}l^{\ast}(\mathbf{X})$ has a choice function; every choice function of $\mathcal{C}l^{\ast}(\mathbf{X})$ is called a \emph{Loeb function} of $\mathbf{X}$.

\item[(ii)] $\mathbf{X}$ is called \emph{s-Loeb} if $\mathcal{SC}^{\ast}(\mathbf{X})$ has a choice function; every choice function of $\mathcal{SC}^{\ast}(\mathbf{X})$ is called an \emph{s-Loeb function} of $\mathbf{X}$.

\item[(iii)] $\mathbf{X}$ is called \emph{$\mathcal{K}$-Loeb} if $\mathcal{K}^{\ast}(\mathbf{X})$ has a choice function.

\item[(iv)] $\mathbf{X}$ is called  \emph{$\mathcal{F}_{\sigma}$-Loeb} (respectively, $\mathcal{G}_{\delta}$-\emph{Loeb}) if the family $\mathcal{F}_{\sigma}^{\ast}(\mathbf{X})$ (respectively, $\mathcal{G}_{\delta}^{\ast}(\mathbf{X})$) has a choice function.

\item[(v)] $\mathbf{X}$ is called  \emph{sequential} if $\mathcal{SC}l(\mathbf{X})=\mathcal{C}l(\mathbf{X})$.

\item[(vi)] $\mathbf{X}$ is called  \emph{Fr\'echet-Urysohn} if, for every subset $A$ of $X$ and every point $x\in\cl_{\mathbf{X}}(A)$, there exists a sequence of points of $A$ converging in $\mathbf{X}$ to the point $x$.

\item[(vii)] $\mathbf{X}$ is called a \emph{$k$-space} if  $\mathbf{X}$ is Hausdorff and, for every set $A\subseteq X$ such that, for every $K\in\mathcal{K}(\mathbf{X})$, $A\cap K\in\mathcal{C}l(\mathbf{X})$, it holds that $A\in\mathcal{C}l(\mathbf{X})$.

\item[(viii)] $\mathbf{X}$ is called a \emph{very $k$-space} (see \cite{ar}) or a \emph{hereditarily $k$-space} if every subspace of $\mathbf{X}$ is a $k$-space.
\end{enumerate}
\end{definition}

Let $\langle X, d\rangle$ be a metric space. For a point $x\in X$ and a positive real number $r$, the set $B_d(x, r)=\{y\in X: d(x,y)<r\}$ is the \emph{open ball with center $x$ and radius $r$}. The family 
$$\tau(d)=\{V\in\mathcal{P}(X): (\forall x\in V)(\exists r\in (0, +\infty)) B_d(x, r)\subseteq V\}$$
\noindent is the \emph{topology on $X$ induced by $d$}. Then $\mathbf{X}=\langle X, \tau(d)\rangle$ is the \emph{topological space associated with} $d$. For a set $A\subseteq X$, $\delta_d(A)=\sup\{d(x,y): x,y\in A\}$ is the \emph{d-diameter} of $A$. 

\begin{definition}
\label{s1d4}
\begin{enumerate}
\item[(i)] A metric $d$ on a set $X$ is called \emph{Cantor complete} if, for every descending family $\mathcal{F}=\{F_n: n\in\omega\}$ of non-empty closed subsets of $\langle X, \tau(d)\rangle$ with $\lim\limits_{n\to+\infty}\delta_d(F_n)=0$, it holds that  $\bigcap\mathcal{F}\neq\emptyset$.
\item[(ii)] A topological space $\langle X, \tau\rangle$ is called \emph{Cantor completely metrizable} if there exists a Cantor complete metric $d$ on $X$ such that $\tau=\tau(d)$.
\item[(iii)] A metric space $\langle X, d\rangle$ is called \emph{Cantor complete} if $d$ is Cantor complete. 
\end{enumerate}
\end{definition}

In the sequel, if it is not stated otherwise,  topological and metric spaces will be denoted by boldface letters, and their underlying sets by lightface letters.  For $n\in(\omega+1)\setminus\{0\}$, the space $\mathbb{R}^{n}$ will be considered with its natural (standard) topology. The space $\mathbb{R}^1$ is identified with $\mathbb{R}$. We denote by $\mathbb{P}$ the space of all irrational numbers as a subspace of $\mathbb{R}$. It is well known that $\mathbb{P}$ is homeomorphic with the Baire space $\mathbb{N}^{\omega}$. 

\subsection{The list of forms relevant to choice principles}
\label{s1.3}

Below, we list the choice principles we use in this paper. For the known forms which can be found in \cite{hr}, we give the form number under which they are recorded in \cite{hr}.

\begin{definition}
\label{s1d5}
\begin{enumerate}
\item $\mathbf{AC}$ (\cite[Form 1]{hr}): Every infinite family of non-empty sets has a choice function.


\item $\mathbf{AC}_{fin}$ (\cite[Form 62]{hr}): Every infinite family of non-empty finite sets has a choice function.

\item $\mathbf{CMC}$ (\cite[Form 126]{hr}): Every denumerable family of non-empty sets has a multiple choice function.

\item $\mathbf{CAC}$ (\cite[Form 8]{hr}): Every denumerable family of non-empty sets has a choice function.

\item $\mathbf{CAC}_{\omega}$ (\cite[Form 32]{hr}): Every denumerable family of non-empty countable sets has a choice function.

\item $\mathbf{CAC}_{fin}$ (\cite[Form 10]{hr}): Every denumerable family of non-empty finite sets has a choice function.

\item $\mathbf{CUC}$ (\cite[Form 31]{hr}): For every countable family $\mathcal{A}$ of countable sets, the union $\bigcup\mathcal{A}$ is a countable set.

\item $\mathbf{IDI}$ (\cite[Form 9]{hr}): Every infinite set is Dedekind-infinite. 
\end{enumerate}
\end{definition}

\begin{definition} 
\label{s1d6}
 Let $X$ be a set. If $\mathbf{\Phi}$ is a form given in Definition \ref{s1d5}, then $\mathbf{\Phi}(\mathbf{X})$ means $\mathbf{\Phi}$ restricted to $X$ or to families of subsets of $\mathbf{X}$, respectively. In particular:
\begin{enumerate}
\item $\mathbf{AC}(\mathbb{R})$ (\cite[Form 79]{hr}): Every infinite family of non-empty subsets of $\mathbb{R}$ has a choice function.
\item $\mathbf{CAC}(\mathbb{R})$ (\cite[Form 94]{hr}): Every denumerable family of non-empty subsets of $\mathbb{R}$ has a choice function.
\item $\mathbf{CAC}_{\omega}(\mathbb{R})$ (\cite[Form 5]{hr}): Every denumerable family of countable non-empty subsets of $\mathbb{R}$ has a choice function.
\item $\omega-\mathbf{CAC}(\mathbb{R})$ (Definition 4.56 in \cite{her}): For every family $\{A_i: i\in\omega\}$ of non-empty subsets of $\mathbb{R}$, there exists a family $\{C_i: i\in\omega\}$ of countable non-empty subsets of $\mathbb{R}$ such that, for each $i\in\omega$, $C_i\subseteq A_i$.
\item $\mathbf{CUC}(\mathbb{R})$ (\cite[Form 6]{hr}): Every denumerable union of countable subsets of $\mathbb{R}$ is countable.
\item $\mathbf{IDI}(\mathbb{R})$ (\cite[Form 13]{hr}): Every infinite subset of $\mathbb{R}$ is Dedekind-infinite.
\end{enumerate}
\end{definition}

Furthermore, for every set $X$,  we introduce four additional forms in the following definition:

\begin{definition}
\label{s1d7}
\begin{enumerate}
\item $\mathbf{KW}(X, \omega)$: For every infinite family $\mathcal{A}$ of non-empty subsets of $X$, the family $\{[A]^{\leq\omega}\setminus\{\emptyset\}: A\in\mathcal{A}\}$ has a choice function. 

\item $\mathbf{CKW}(X, \omega)$: $\mathbf{KW}(X, \omega)$ restricted to denumerable families of non-empty subsets of $X$.

\item $\mathbf{PCAC}(X)$: Every denumerable family of non-empty subsets of $X$ has a partial choice function.

\item $\mathbf{PCMC}(X)$: Every denumerable family of non-empty subsets of $X$ has a partial multiple choice function. 
\end{enumerate}
\end{definition}

We define below several new statements concerning topological spaces and recall some known ones. 

\begin{definition} 
\label{s1d8}
Let $\mathbf{X}=\langle X, \tau\rangle$ be a topological space.
\begin{enumerate}

\item $\mathbf{S}(\mathbf{X})$: $\mathbf{X}$ is sequential.

\item $\mathbf{HS}(\mathbf{X})$: Every subspace of $\mathbf{X}$ is sequential.

\item $\mathbf{FU}(\mathbf{X})$: $\mathbf{X}$ is Fr\'echet-Urysohn.

\item $\mathbf{k}(\mathbf{X})$: $\mathbf{X}$ is a $k$-space.

\item $\mathbf{Hk}(\mathbf{X})$: Every subspace of $\mathbf{X}$ is a $k$-space.

\item $\mathbf{L}(\mathbf{X})$: $\mathbf{X}$ is Loeb.

\item $\mathbf{SL}(\mathbf{X})$: $\mathbf{X}$ is s-Loeb.

\item $\mathbf{KL}(\mathbf{X})$: $\mathbf{X}$ is $\mathcal{K}$-Loeb.

\item $\mathbf{F_{\sigma}L}(\mathbf{X})$: $\mathbf{X}$ is $\mathcal{F}_{\sigma}$-Loeb.

\item $\mathbf{G_{\delta}L}(\mathbf{X})$: $\mathbf{X}$ is $\mathcal{G}_{\delta}$-Loeb.

\item $\mathbf{CF}_{\sigma}\mathbf{L}(\mathbf{X})$: Every countable family of non-empty $F_{\sigma}$-subsets of $\mathbf{X}$ has a choice function.

\item $\mathbf{CG}_{\delta}\mathbf{L}(\mathbf{X})$: Every countable family of non-empty $G_{\delta}$-subsets of $\mathbf{X}$ has a choice function.

\item $\mathbf{\Pi F}_{\sigma}(\mathbf{X})$: Every infinite subset of $X$ contains an infinite $F_{\sigma}$-set of $\mathbf{X}$.

\item $\mathbf{\Pi G}_{\delta}(\mathbf{X})$:  Every infinite subset of $X$ contains an infinite $G_{\delta}$-set of $\mathbf{X}$. 

\item $\mathbf{UIC}(\mathbf{X})$: Every uncountable subset of $X$ contains an infinite closed subset of $\mathbf{X}$.

\item $\mathbf{UISC}(\mathbf{X})$: Every uncountable subset of $X$ contains an infinite sequentially closed set of $\mathbf{X}$.

\item $\mathbf{ISCIC}(\mathbf{X})$: Every infinite sequentially closed subset of $\mathbf{X}$ contains an infinite closed subset of $\mathbf{X}$.

\item $\mathbf{SCF_{\sigma}}(\mathbf{X})$: Every sequentially closed subset of $\mathbf{X}$ is of type $F_{\sigma}$ in $\mathbf{X}$.

\item $\mathbf{SCDF_{\sigma}}(\mathbf{X})$: Every sequentially closed subset $A$ of $\mathbf{X}$ contains an $F_{\sigma}$-set $D$ of $\mathbf{X}$ such that $D$ is dense in the subspace $\mathbf{A}$ of $\mathbf{X}$.

\item $\mathbf{SCDWOU}(\mathbf{X})$: Every sequentially closed subset $A$ of $\mathbf{X}$ contains a set $D$ such that $D$ is expressible as a well-ordered union of well-orderable sets, and $D$ is dense in the subspace $\mathbf{A}$ of $\mathbf{X}$.

\item $\mathbf{WUF}(\mathbf{X})$: For every closed subset $F$ of $\mathbf{X}$ and every accumulation point $x$ of $F$ in $\mathbf{X}$, there exists a sequence of points of $F\setminus\{x\}$ which converges in $\mathbf{X}$ to the point $x$.

\item $\mathbf{CUF}(\mathbf{X})$: For every compact subset $F$ of $\mathbf{X}$ and every accumulation point $x$ of $F$ in $\mathbf{X}$, there exists a sequence of points of $F\setminus\{x\}$ which converges in $\mathbf{X}$ to the point $x$.

\item $\mathbf{CPCAC}(\mathbf{X})$: For every infinite compact subset $K$ of $\mathbf{X}$, $\mathbf{PCAC}(K)$.

\item $\mathbf{CPCMC}(\mathbf{X})$: For every infinite compact subset $K$ of $\mathbf{X}$, $\mathbf{PCMC}(K)$.
\end{enumerate}
\end{definition}

Forms 4, 5 and 8--24 of Definition \ref{s1d8} are new.

\begin{definition}
\label{s1d9}
If, for a topological space $\mathbf{X}$, $\mathbf{\Phi}(\mathbf{X})$ is a form defined in Definition \ref{s1d8}, then:
\begin{enumerate}
\item $\mathbf{\Phi}$ means: For every topological space $\mathbf{X}$, $\mathbf{\Phi}(\mathbf{X})$ holds.
\item $\mathbf{\Phi}_M$ means: For every metrizable space $\mathbf{X}$, $\mathbf{\Phi}(\mathbf{X})$ holds.
\item $\mathbf{\Phi}_{M2}$ means: For every second-countable metrizable space $\mathbf{X}$, $\mathbf{\Phi}(\mathbf{X})$ holds.
\end{enumerate}
\end{definition}

\subsection{Some known theorems}
\label{s1.4}
Before we pass to the main body of the paper, let us recall several known theorems for future references.

\begin{theorem}
\label{s1t:arkh} 
(Cf. \cite{ar} and, e.g., \cite[Exercise 3.3.1]{en}.) $[\mathbf{ZFC}]$  A Hausdorff space $\mathbf{X}$ is a very $k$-space if and only if $\mathbf{X}$ is Fr\'echet-Urysohn.
\end{theorem}

\begin{proposition}
\label{s1:WOT}
(Cf., e.g.,\cite[proof of Theorem 1.4]{her}.) $[\mathbf{ZF}]$ For every set $X$, $\mathbf{AC}(X)$ is equivalent to the sentence: $X$ is well-orderable.
\end{proposition}

\begin{theorem}
\label{s1t12}
(Cf. \cite[Theorem 2.3 and Corollary 2.4 ]{kw1}.) $[\mathbf{ZF}]$ Every second-countable Cantor completely metrizable space is Loeb. In particular, $\mathbb{R}^{\omega}$, the Cantor cube $\{0,1\}^{\omega}$, the Hilbert cube $[0, 1]^{\omega}$ and the Baire space $\mathbb{N}^{\omega}$, hence the space $\mathbb{P}$ of irrationals also, are Loeb. 
\end{theorem}

\begin{theorem}
\label{s1t13}
$[\mathbf{ZF}]$ Every second-countable metrizable space is $\mathcal{K}$-Loeb.
\end{theorem}

We say that a topological space $\mathbf{X}$ has a cuf base if there exists an (open) base $\mathcal{B}$ of $\mathbf{X}$ such that $\mathcal{B}$ is a cuf set.  It was shown in \cite[Corollary 4.8]{gt} that Urysohn's Metrization Theorem is provable in $\mathbf{ZF}$, that is, every second-countable $T_3$-space is metrizable in $\mathbf{ZF}$. This result was generalized in \cite{ktw1} to the following theorem:

\begin{theorem}
\label{s1t14}
(Cf. \cite[Theorem 2.1]{ktw1}.) $[\mathbf{ZF}]$  Every $T_3$-space which has a cuf base is metrizable. 
\end{theorem}

\begin{theorem}
\label{s1t15}
$[\mathbf{ZF}]$ (Cf. \cite[Section 4.6, Exercise E2]{hr}.) If $D$ is an infinite Dedekind-finite subset of $\mathbb{R}$, then $D$ is not closed in $\mathbb{R}$.
\end{theorem}

\begin{theorem}
\label{s1t16}
(Cf. \cite[Theorem 4.54, and Exercise E3 of Section 4.6]{hr}.) $[\mathbf{ZF}]$ Each of $\mathbf{CAC}(\mathbb{R})$, $\mathbf{HS}(\mathbb{R})$, $\mathbf{FU}(\mathbb{R})$, $\mathbf{FU}_{M2}$, $\mathbf{S}_{M2}$ is equivalent to each of following:
\begin{enumerate}
\item[(i)] for every second-countable $T_0$-space $\mathbf{X}$, $\mathbf{FU}(\mathbf{X})$ holds;
\item[(ii)] every second-countable $T_0$-space is separable;
\item[(iii)]  every second-countable metrizable space is separable;
\item[(iv)] $\mathbb{R}$ is hereditarily separable;
\item[(v)] every second-countable $T_0$-space is sequential;
\item [(vi)] every second-countable Cantor completely metrizable space is Fr\'echet-Urysohn.
\end{enumerate}
\end{theorem}

\begin{theorem} 
\label{s1:fil}
(Cf. \cite[Theorem 3.3]{kopsw}.) $[\mathbf{ZF}]$ The following are equivalent:
\begin{enumerate}
\item[(i)] $\mathbf{CMC}$;
\item[(ii)] every first-countable Hausdorff space is a $k$-space;
\item[(iii)] every metrizable space is a $k$-space.
\end{enumerate}
\end{theorem}

\begin{proposition}
\label{s1p18}
$[\mathbf{ZF}]$
\begin{enumerate}
\item[(i)] For every topological space $\mathbf{X}$, $\mathbf{FU}(\mathbf{X})$ and $\mathbf{HS}(\mathbf{X})$ are equivalent.
\item[(ii)] Every sequential Hausdorff space is a $k$-space. 
\end{enumerate}
\end{proposition}

One can easily verify that item (i) of Proposition \ref{s1p18} is true in $\mathbf{ZF}$. To check that item (ii) of Proposition \ref{s1p18} is also true in $\mathbf{ZF}$, one can mimic the proof of Theorem 3.3.20 in \cite{en}. Although the following very useful proposition is not especially original, our formulation is more general than the relevant known ones  (see, e.g., \cite[Proposition 1]{k2}), so we include its proof for completeness and future references. 

\begin{proposition}
\label{s1p19}
$[\mathbf{ZF}]$  Let $\mathbf{X}$ be a first-countable space. 
\begin{enumerate}
\item[(i)] Let $A\in\mathcal{SC}(\mathbf{X})$. If the subspace $\mathbf{A}$ of $\mathbf{X}$ contains a dense well-orderable set, then $A\in\mathcal{C}l(\mathbf{X})$.
\item[(ii)] $\mathbf{S}(\mathbf{X})$ if and only if, for every $A\in\mathcal{SC}(\mathbf{X})$ and every $x\in\cl_{\mathbf{X}}(A)$, there exists a well-orderable subset $F$ of $A$ such that $x\in\cl_{\mathbf{X}}(F)$.
\end{enumerate}
\end{proposition}

\begin{proof} Let $\mathcal{B}(x)=\{U_n: n\in \omega\}$ be a base of neighborhoods of $x$ in $\mathbf{X}$ such that, for every $n\in\omega$, $U_{n+1}\subseteq U_n$.  Let $A\subseteq X$ and $x\in\cl_{\mathbf{X}}(A)$. Suppose that $F$ is a well-orderable subset of $A$ such that $x\in\cl_{\mathbf{X}}(F)$. We fix a well-ordering $\leq$ on $F$ and consider any $n\in\omega$. Since $x\in\cl_{\mathbf{X}}(F)$, the set $F\cap U_{n}$ is non-empty. We can define $x_n$ as the first element of $F\cap U_n$ in $\langle F,\leq\rangle$. Then $(x_n)_{n\in\omega}$ is a sequence of points of $A$ which coverges to $x$ in $\mathbf{X}$. Hence, if $A\in\mathcal{SC}(\mathbf{X})$, then $x\in A$.  This implies that both (i) and the sufficiency part of (ii) hold. To complete the proof, it suffices to notice that if $\mathbf{S}(\mathbf{X})$ and $A\in\mathcal{SC}(\mathbf{X})$, then $\cl_{\mathbf{X}}(A)=A$, so, for $x\in\cl_{\mathbf{X}}(A)$, the set $F=\{x\}$ is a well-orderable subset of $A$ such that $x\in\cl_{\mathbf{X}}(F)$.
\end{proof}

Basic facts about sequentially closed sets and sequential spaces in $\mathbf{ZF}$ can be found, for instance, in \cite[Section 4.6]{hr}, \cite{k2},  \cite{kw1} and \cite{kw2}.

\subsection{The content of the article in brief}
\label{s1.5}

Among other results of Section \ref{s2}, we prove in $\mathbf{ZF}$ that if $\mathbf{X}$ is a Loeb $T_3$-space which has a cuf base, then the family of all non-empty open subsets of $\mathbf{X}$ has a choice function, and  $\mathbf{X}$ is a metrizable, second-countable space whose every $F_{\sigma}$-subspace is separable (see Theorem \ref{s2t3}). We show in $\mathbf{ZF}$ that every $F_{\sigma}$-subspace of a Loeb space is Loeb (see Proposition \ref{s2p1}), and if $\mathbf{X}$ is a locally compact, $\mathcal{K}$-Loeb Hausdorff space which has a cuf base, then every $G_{\delta}$-subspace of $\mathbf{X}$ is Cantor completely metrizable and Loeb (see Theorem \ref{s2t6} and Corollary \ref{s2c8}). We notice in Theorem \ref{s2t7} that it is provable in $\mathbf{ZF}$ that every $G_{\delta}$-subspace of a Cantor completely metrizable space is Cantor completely metrizable. In Theorem \ref{s2t9}, we show in $\mathbf{ZF}$ that every $G_{\delta}$-subspace of a  second-countable, Cantor completely metrizable space is Loeb and separable.  We deduce that it holds in $\mathbf{ZF}$ that if $\mathbf{X}$ is a second-countable Cantor completely metrizable space, then every infinite $G_{\delta}$ or $F_{\sigma}$-subset of $\mathbf{X}$ is Dedekind-infinite (see Corollary \ref{s2c10}). We prove in $\mathbf{ZF}$ that every countably compact $\mathcal{K}$-Loeb space is compact, every scattered $\mathcal{K}$-Loeb $T_1$-space which has a cuf base is well-orderable, and every Cantor completely metrizable $\mathcal{K}$-Loeb space whose every base contains a cuf base is Loeb. We notice that, in a model of $\mathbf{ZF}$, a Loeb $T_1$-space may fail to be $\mathcal{K}$-Loeb (see Proposition \ref{s2p14}). 

In Section \ref{s3}, we notice that  it holds in $\mathbf{ZF}$ that if $\mathbf{X}$ is a second-countable Cantor completely metrizable space, then every sequentially closed subset of type $G_{\delta}$ or $F_{\sigma}$ in $\mathbf{X}$ is closed (see Proposition \ref{s3p1}) and, in consequence, the families $\mathcal{G}_{\delta}^{\ast}(\mathbf{X})\cap\mathcal{SC}(\mathbf{X})$ and $\mathcal{F}_{\sigma}^{\ast}(\mathbf{X})\cap\mathcal{SC}(\mathbf{X})$ have choice functions (see Corollary \ref{s3c3}). We prove in $\mathbf{ZF}$ that, for every second-countable metrizable Loeb space $\mathbf{X}$, the statements $\mathbf{S}(\mathbf{X})$, $\mathbf{SL}(\mathbf{X})$, $\mathbf{SCF}_{\sigma}(\mathbf{X})$, $\mathbf{SCDF}_{\sigma}(\mathbf{X})$ and $\mathbf{SCDWOU}(\mathbf{X})$ are all equivalent (see Theorem \ref{s3t4}).

Section \ref{s4} concerns mainly forms (9)--(17) from Definition \ref{s1d8}. In particular, we prove in $\mathbf{ZF}$ that, for every Cantor completely metrizable second-countable space $\mathbf{X}$, the forms $\mathbf{IDI}(\mathbf{X})$, $\mathbf{\Pi G}_{\delta}(\mathbf{X})$, $\mathbf{\Pi F}_{\sigma}(\mathbf{X})$ and $\mathbf{ISCIC}(\mathbf{X})$ are all equivalent (see Corollary \ref{s4c2}). We notice in Proposition \ref{s4p3} that, in $\mathbf{ZF}$, $\mathbf{AC}$ (respectively, $\mathbf{CAC}$) is equivalent to $\mathbf{F_{\sigma}L}_M$ and $\mathbf{G_{\delta}L}_M$ (respectively,  $\mathbf{CF_{\sigma}L}_M$ and $\mathbf{CG_{\delta}L}_M$). In Theorem \ref{s4t4}, we show that if $\mathbf{X}$ is a second-countable Cantor completely metrizable space, then it holds in $\mathbf{ZF}$ that $\mathbf{UIC}(\mathbf{X})$ is equivalent to each of the conjunctions $\mathbf{ISCIC}(\mathbf{X})\wedge\mathbf{UISC}(\mathbf{X})$ and $\mathbf{IDI}(\mathbf{X})\wedge\mathbf{UISC}(\mathbf{X})$. We also prove that it holds in $\mathbf{ZF}$ that, for every second-countable Hausdorff space $\mathbf{X}$, $\mathbf{UIC}(\mathbf{X})$ follows from each of $\mathbf{FU}(\mathbf{X})$ and $\omega-\mathbf{CAC}(\mathbf{X})$ (see Corollary \ref{s4c6} and Proposition \ref{s4p7}). In Theorem \ref{s4t10}, we show that it holds in $\mathbf{ZF}$ that, for every $T_1$-space $\mathbf{X}$, $\mathbf{AC}(X)$ (respectively, $\mathbf{CAC}(X)$), is equivalent to the conjunction $\mathbf{F_{\sigma}L}(\mathbf{X})\wedge\mathbf{KW}(X, \omega)$ (respectively, $\mathbf{CF_{\sigma}L}(\mathbf{X})\wedge\mathbf{CKW}(X, \omega)$). Other relevant results are also included in Section \ref{s4}.

In Section 5, we turn attention to condensation points to find infinite closed subsets of some uncountable sequentially closed sets in metrizable spaces with cuf bases.  Among other facts, we show in Theorem \ref{s5t1} that, in $\mathbf{ZF}$, $\mathbf{CUC}$ is equivalent to the sentence ``For every second-countable space $\mathbf{X}$ and every uncountable subset $A$ of $X$, there exists a condensation point $x$ of $A$ in $\mathbf{X}$ such that $x\in A$''. We deduce that if a subset $A$ of a topological space $\mathbf{X}$ with a cuf base has no condensation points in $\mathbf{X}$, then $A$ is a countable union of countable sets (see Proposition \ref{s5p2}).  We notice in Theorem \ref{s5t3} that the sentence ``For every metrizable space $\mathbf{X}$ with a cuf base, it holds that every uncountable subset of $X$ has a condensation point in $\mathbf{X}$'' is unprovable in $\mathbf{ZF}$ for it implies $\mathbf{CAC}_{fin}$. Proposition \ref{s5t4} shows that it holds in $\mathbf{ZF}$ that $\mathbf{CAC}_{fin}$ implies that every uncountable sequentially closed set $A$ of a metrizable space $\mathbf{X}$ with a cuf base such that no point of $A$ is a condensation point of $A$ contains an infinite closed set of $\mathbf{X}$.

Now, let us mention several (but not all) results of Section  \ref{s6}. Corollary \ref{s6c4} shows that Arkhangel'skii's statement ``Every very $k$-space is Fr\'echet-Urysohn'' (see Theorem \ref{s1t:arkh}) is unprovable in $\mathbf{ZF}$. We prove that it holds in $\mathbf{ZF}$ that if $\mathbf{X}$ is a first-countable, $\mathcal{K}$-Loeb $T_3$-space, then $\mathbf{Hk}(\mathbf{X})$ and $\mathbf{FU}(\mathbf{X})$ are equivalent (see Theorem \ref{s6t2}). Theorem \ref{s6t5} shows that, in $\mathbf{ZF}$,  the forms $\mathbf{CAC}(\mathbb{R})$, $\mathbf{Hk}_{M2}$ and $\mathbf{k}_{M2}$ are all equivalent, and $\mathbf{k}_{M2}$ implies $\mathbf{UIC}_{M2}$. We prove in $\mathbf{ZF}$ that the sentence ``Every subspace of $\mathbb{R}$ is either a $k$-space or Loeb'' implies  $\mathbf{CAC}(\mathbb{R})$, and the negation of $\mathbf{CAC}_{\omega}(\mathbb{R})$ implies that there exists a subspace $\mathbf{X}$ of $\mathbb{R}$ such that $\mathbf{X}$ is neither a $k$-space nor Loeb but $\mathbf{UIC}(\mathbf{X})$ is true (see Theorem \ref{s6t8}). We show that the sentence ``Every first-countable Hausdorff space is Fr\'echet-Urysohn'' is equivalent to $\mathbf{CAC}$ in $\mathbf{ZF}$, but it fails to be  equivalent to the sentence ``Every first-countable Hausdorff space is a $k$-space'' in $\mathbf{ZFA}$ (see Theorem \ref{s6t12} and Remark \ref{s6r13}). We notice in Proposition \ref{s6p14} that, in $\mathbf{ZF}$,  for every $k$-space $\mathbf{X}$, $\mathbf{WUF}(\mathbf{X})$ and $\mathbf{CUF}(\mathbf{X})$ are equivalent, and for every space $\mathbf{X}$, $\mathbf{S}(\mathbf{X})$ implies $\mathbf{WUF}(\mathbf{X})$. We prove that, in $\mathbf{ZF}$,  $\mathbf{WUF}_M$ and $\mathbf{CAC}$ are equivalent, $\mathbf{CUF}_M$ implies $\mathbf{CAC}_{fin}$ and this implication is not reversible; furthermore, if $\mathbf{IDI}(\mathbb{R})$ is false, then there exists a metrizable non-sequential space $\mathbf{X}$ for which $\mathbf{WUF}(\mathbf{X})$ holds (see Theorem \ref{s6t15}).

Theorems \ref{s7t1} and \ref{s7t7} are the main results of Section \ref{s7}. It is shown in Theorem \ref{s7t1} that it is true in $\mathbf{ZF}$ that a first-countable Hausdorff space $\mathbf{X}$ is Fr\'echet-Urysohn if and only if $\mathbf{X}$ is a very $k$-space for which $\mathbf{CPCAC}(\mathbf{X})$ holds. Theorem \ref{s7t7} shows that it is true in $\mathbf{ZF}$ that, for every first-countable $\mathcal{K}$-Loeb $T_3$-space $\mathbf{X}$, the statements $\mathbf{FU}(\mathbf{X})$, $\mathbf{Hk}(\mathbf{X})$ and $\mathbf{CPCAC}(\mathbf{X})$ are all equivalent. 

For the convenience of readers, we include a shortlist of open problems in Section \ref{s8}.

\section{$F_{\sigma}$ and $G_{\delta}$-subspaces of Loeb spaces}
\label{s2}

In this section, we show several properties concerning subspaces of Loeb spaces.

\begin{proposition}
\label{s2p1}
$[\mathbf{ZF}]$
Every $F_{\sigma}$-subspace of a Loeb space is Loeb.
\end{proposition}

\begin{proof}
Suppose that $P$ is an $F_{\sigma}$-subset of a Loeb space $\mathbf{X}$. Let $P=\bigcup_{n\in\omega}P_n$ where, for every $n\in\omega$, the set $P_n$ is closed in $\mathbf{X}$. To show that the subspace $\mathbf{P}$ of $\mathbf{X}$ is Loeb, we fix a Loeb function $f$ of $\mathbf{X}$ and consider an arbitrary set $A\in\mathcal{C}l^{\ast}(\mathbf{P})$. Let $m(A)=\min\{n\in\omega: P_n\cap A\neq\emptyset\}$. Since $A\cap P_{m(A)}\in\mathcal{C}l^{\ast}(\mathbf{X})$, we can define $f_P(A)= f(A\cap P_{m(A)})$. In this way, we have defined a Loeb function $f_P$ of $\mathbf{P}$.
\end{proof}
 
\begin{proposition}
\label{s2p2}
$[\mathbf{ZF}]$ Let $\mathbf{X}$ be a $T_1$-space. If $\mathbf{X}$ is $\mathcal{K}$-Loeb or second-countable,  then every cuf subset of $X$ is countable.
\end{proposition}
\begin{proof}
Let $E$ be a non-empty finite subset of $X$ such that $E$ is not a singleton. We assume that $\mathbf{X}$ is either $\mathcal{K}$-Loeb or second-countable. We are going to define a well-ordering $\leq_E$ on $E$.

First, suppose that $\mathbf{X}$ is $\mathcal{K}$-Loeb. Let $f$ be a choice function of $\mathcal{K}^{\ast}(\mathbf{X})$. Let $n_E\in\omega$ be equipotent with $E$.  We effectively define a bijection $\psi_E:n_E\to E$ as follows. We put $\psi_E(0)=f(E)$.  Suppose that $n\in n_E$ is such that, for every $i\in n$,  $\psi_E(i)$ has been defined. Then we put $\psi_E(n)=f(E\setminus\{\psi_E(i): i\in n\})$. Having defined $\psi_E$, we can define a well ordering $\leq_E$ on $E$ as follows. For $x,y\in E$, we put $x\leq_E y$ if and only if $\psi_E^{-1}(x)\subseteq \psi_E^{-1}(y)$. 

Now, suppose that $\mathbf{X}$ is second-countable. Let $\mathcal{B}=\{U_n: n\in\omega\}$ be a base for $\mathbf{X}$ such that $\mathcal{B}$ is stable under finite intersections. For every $x\in E$, the set $N(x)=\{n\in\omega: x\in U_n\wedge U_n\cap (E\setminus\{x\})=\emptyset\}$ is non-empty, so we can define $n(x)=\min N(x)$. Now, we can define a well-ordering $\leq_E$  on $E$ as follows. For $x,y\in E$, 
$$x\leq_E y\leftrightarrow n(x)\subseteq n(y).$$  

Let $\{A_n: n\in\omega\}$ be a family of non-empty finite subsets of $\mathbf{X}$. We have shown that there exists a family $\{\leq_n: n\in\omega\}$ such that, for every $n\in\omega$, $\leq_n$ is a well-ordering on $A_n$. This implies that $\bigcup_{n\in\omega}A_n$ is countable as a countable union of well-ordered finite sets.
\end{proof}

\begin{theorem}
\label{s2t3}
$[\mathbf{ZF}]$ Let $\mathbf{X}=\langle X, \tau\rangle$ be a non-empty regular Loeb space which has a cuf base. Then the family $\tau\setminus\{\emptyset\}$ has a choice function. In consequence, if $\mathbf{X}$ is also a $T_1$-space, then $\mathbf{X}$ is a second-countable metrizable space such that every $F_{\sigma}$-subspace of $\mathbf{X}$ is separable. 
\end{theorem}
\begin{proof}
Suppose that $\mathcal{B}=\bigcup_{n\in\omega}\mathcal{B}_n$ is a base of $\mathbf{X}$ such that, for every $n\in\omega$, the family $\mathcal{B}_n$ is finite. We may assume that $\emptyset\notin\mathcal{B}$. Let $f$ be a Loeb function of $\mathbf{X}$. We define a function $\psi\in\prod_{V\in\tau\setminus\{\emptyset\}}V$ as follows. For every $V\in\tau\setminus\{\emptyset\}$, let $n_V=\min\{n\in\omega:  (\exists U\in\mathcal{B}_n) \cl_{\mathbf{X}}(U)\subseteq V\}$,  $U(V)=\bigcup\{U\in\mathcal{B}_{n_V}: \cl_{\mathbf{X}}(U)\subseteq V\}$ and $\psi(V)=f(\cl_{\mathbf{X}}(U(V)))$. The set $D=\{\psi(V): V\in \mathcal{B}\}$ is dense in $\mathbf{X}$ and $D$ is a cuf set. Now, let us assume that $\mathbf{X}$ is also a $T_1$-space. It follows from Proposition \ref{s2p2} that the set $D$ is countable, hence $\mathbf{X}$ is separable. By Theorem \ref{s1t14}, $\mathbf{X}$ is metrizable. It is known that every metrizable separable space is second-countable, so $\mathbf{X}$ is second-countable.

Let $P=\bigcup_{n\in\omega}P_n$ where, for every $n\in\omega$, $P_n\in\mathcal{C}l(\mathbf{X})$. The subspace $\mathbf{P}$ of $\mathbf{X}$ is regular and has a cuf base. By Proposition \ref{s2p1}, $\mathbf{P}$ is Loeb. It follows from the first part of the proof that $\mathbf{P}$ has a dense cuf set $E$ and, if $\mathbf{P}$ is a $T_1$-space, then $E$ is countable by Proposition \ref{s2p2}. This completes the proof.
\end{proof}

\begin{remark}
\label{s2r4}
(i) It follows from Theorem \ref{s1t12} and  Theorem \ref{s2t3} that $\mathbb{P}$ is both Loeb and separable in $\mathbf{ZF}$. However, that $\mathbb{P}$ is separable in $\mathbf{ZF}$ has a very elementary proof. Namely, it is easily seen that the set $\{q+\frac{\sqrt{2}}{n}: q\in\mathbb{Q}\wedge n\in\mathbb{N}\}$ is a countable dense subset of $\mathbb{P}$. \medskip

(ii) Metric spaces for which the family of all non-empty open sets have a choice function are called \emph{selective}, for instance, in \cite{k01}, \cite{k1} and \cite{kt1}.
\end{remark}

In \cite{am},  A. Miller showed a model of $\mathbf{ZF}$ in which an uncountable Dede\-kind-finite subset of $\mathbb{R}$ is of type $F_{\sigma\delta}$ in $\mathbb{R}$. One can deduce from Theorem \ref{s1t15} that infinite Dedekind-finite subsets of $\mathbb{R}$ are not of type $F_{\sigma}$. With Theorem \ref{s2t3}, we can generalize this result as follows:

\begin{theorem}
\label{s2t5}
$[\mathbf{ZF}]$ Let $\mathbf{X}$ be a Loeb $T_3$-space which has a cuf base. Then no infinite Dedekind-finite subset of $X$ is of type $F_{\sigma}$ in $\mathbf{X}$.
\end{theorem}
\begin{proof}
It suffices to observe that, by Theorem \ref{s2t3}, if $\mathbf{P}$ is an infinite $F_{\sigma}$-subspace of $\mathbf{X}$, then $P$ is Dedekind-infinite because $P$ contains a denumerable set which is dense in $\mathbf{P}$.
\end{proof}

\begin{theorem}
\label{s2t6}
$[\mathbf{ZF}]$ Let $\mathbf{X}$ be a locally compact Hausdorff $\mathcal{K}$-Loeb space which has a cuf base. Then every $G_{\delta}$-subspace of $\mathbf{X}$ is Loeb.
\end{theorem}
\begin{proof}
Let $F=\bigcup_{n\in\omega}F_n$ where, for every $n\in\omega$, $F_n\in\mathcal{C}l(\mathbf{X})$. Let $P=X\setminus F$. To show that the subspace $\mathbf{P}$ of $\mathbf{X}$ is Loeb, we fix a choice function $f$ of $\mathcal{K}^{\ast}(\mathbf{X})$ and a base $\mathcal{B}=\bigcup_{i\in\omega}\mathcal{B}_i$ of $\mathbf{X}$ such that, for every $i\in\omega$, the family $\mathcal{B}_i$ is finite and, for every $U\in\mathcal{B}_i$, $\cl_{\mathbf{X}}(U)\in\mathcal{K}(\mathbf{X})$. Let $A\in\mathcal{C}l^{\ast}(\mathbf{P})$. The set $I_0=\{i\in\omega: (\exists U\in\mathcal{B}_i) (U\cap A\neq\emptyset\wedge F_0\cap\cl_{\mathbf{X}}(U)=\emptyset)\}$ is non-empty, so we can define $n_0=\min I_0$ and $U_0=\bigcup\{U\in\mathcal{B}_{n_0}: U\cap A\neq\emptyset\wedge F_0\cap\cl_{\mathbf{X}}(U)=\emptyset\}$.  Suppose that, for $k\in\omega$, we have already defined $n_k\in\omega$ and an open subset $U_k$ of $\mathbf{X}$, such that $U_{k}\cap A\neq\emptyset$ and $F_k\cap\cl_{\mathbf{X}}(U_k)=\emptyset$.  Then the set $I_{k+1}=\{i\in\omega: (\exists U\in\mathcal{B}_{i})( U\cap A\neq\emptyset\wedge \cl_{\mathbf{X}}(U)\subseteq  U_{k}\wedge F_{k+1}\cap \cl_{\mathbf{X}}(U)=\emptyset)\}$ is non-empty. We define $n_{k+1}=\min I_{k+1}$ and $U_{k+1}=\bigcup\{U\in\mathcal{B}_{n_{k+1}}: U\cap A\neq\emptyset\wedge \cl_{\mathbf{X}}(U)\subseteq  U_{k}\wedge F_{k+1}\cap \cl_{\mathbf{X}}(U)=\emptyset\}$.  In this way, we have defined by induction the sequence $(U_k)_{k\in\omega}$ of open sets of $\mathbf{X}$. Let $D_A=\bigcap_{k\in\omega}\cl_{\mathbf{X}}(U_{k})\cap\cl_\mathbf{X}(A)$. Since $\mathbf{X}$ is Hausdorff, it follows from the compactness of the sets $\cl_{\mathbf{X}}(U_{k})$ and from our definition of the sequence  $(U_k)_{k\in\omega}$ that the set $D_A$ is a non-empty compact subset of $\mathbf{X}$. Clearly, $D_A\cap F=\emptyset$, so $D_A\subseteq P$. Since $D_A\subseteq\cl_{\mathbf{X}}(A)$ and $A$ is closed in $P$, we deduce that $D_A\subseteq A$. Now, we can define $f_P(A)=f(D_A)$. This defines a Loeb function $f_P$ of $\mathbf{P}$.
\end{proof}

We include a proof of the following theorem for completeness. 

\begin{theorem}
\label{s2t7}
$[\mathbf{ZF}]$ Every $G_{\delta}$-subspace of a Cantor completely metrizable space is Cantor completely metrizable.
\end{theorem}
\begin{proof} Let $\mathbf{X}=\langle X, d\rangle$ be a Cantor complete metric space and let $P$ be a $G_{\delta}$-subset of $\mathbf{X}$. In much the same way, as in the proof of Lemma 4.3.22 in \cite{en}, one can show in $\mathbf{ZF}$ that the subspace $\mathbf{P}$ of $\langle X, \tau(d)\rangle$ is homeomorphic with a closed subspace of $\langle X, \tau(d)\rangle\times\mathbb{R}^{\omega}$. It is known from \cite{k1}  that every countable product of Cantor complete metric spaces is a Cantor complete metric space and a closed metric subspace of a Cantor complete metric space is Cantor complete. This implies that $\mathbf{P}$ is Cantor completely metrizable. 
\end{proof}

\begin{corollary}
\label{s2c8}
$[\mathbf{ZF}]$ Let $\mathbf{X}$ be a locally compact Hausdorff space which has a cuf base, and let $P$ be a $G_{\delta}$-subset of $\mathbf{X}$. Then the subspace $\mathbf{P}$ of $\mathbf{X}$ is Cantor completely metrizable. 
\end{corollary}

\begin{proof}
By Theorem \ref{s1t14}, the space $\mathbf{X}$ is metrizable. Suppose that $\mathbf{X}$ is compact. Then $\mathbf{X}$ is Cantor completely metrizable, so $\mathbf{P}$ is Cantor completely metrizable by Theorem \ref{s2t7}.

Assuming that  $\mathbf{X}$ is non-compact, we can consider the one-point Hausdorff compactification $\alpha\mathbf{X}$  of $\mathbf{X}$. Since $\mathbf{X}$ has a cuf base, the space  $\alpha\mathbf{X}$ has a cuf base. By Theorem \ref{s1t14}, $\alpha\mathbf{X}$ is metrizable.  Clearly, $\alpha\mathbf{X}$ is Cantor completely metrizable as a metrizable compact space. The set $P$ is of type $G_{\delta}$ in $\alpha\mathbf{X}$. This, together with Theorem \ref{s2t7}, implies that $\mathbf{P}$ is Cantor completely metrizable. 
\end{proof}

\begin{theorem}
\label{s2t9}
$[\mathbf{ZF}]$ Let $\mathbf{X}$ be a Cantor completely metrizable space.
\begin{enumerate}
\item[(i)] If $\mathbf{X}$ is second-countable, then every $G_{\delta}$-subspace of $\mathbf{X}$ is Loeb, so also separable. 
\item[(ii)] If $\mathbf{X}$ is Loeb and has a cuf base, then every $G_{\delta}$-subspace of $\mathbf{X}$ is Loeb and separable. 
\end{enumerate}
\end{theorem}
\begin{proof}
That (i) holds follows immediately from Theorems \ref{s2t7} and \ref{s1t12}, and the obvious fact that every second-countable regular Loeb space is separable.

To prove (ii), we observe that if $\mathbf{X}$ is Loeb and has a cuf base, then $\mathbf{X}$ is second-countable by Theorem \ref{s2t3}, so (ii) is a consequence of (i).
\end{proof}

\begin{corollary}
\label{s2c10}
$[\mathbf{ZF}]$ Let $\mathbf{X}$ be a Cantor completely metrizable second-countable space. Then every infinite set $A\in\mathcal{F}_{\sigma}(\mathbf{X})\cup\mathcal{G}_{\delta}(\mathbf{X})$ is Dedekind-infinite. 
\end{corollary}

\begin{proof}
Let $A\in\mathcal{F}_{\sigma}(\mathbf{X})\cup\mathcal{G}_{\delta}(\mathbf{X})$ be infinite. If  $A\in\mathcal{F}_{\sigma}(\mathbf{X})$, it follows from Theorems \ref{s2t5} and \ref{s1t12} that $A$ is Dedekind-infinite. If $A\in\mathcal{G}_{\delta}(\mathbf{X})$, then, by Theorem \ref{s2t9}, the subspace $\mathbf{A}$ of $\mathbf{X}$ is separable, so $A$ is Dedekind-infinite. 
\end{proof}

\begin{remark} 
\label{s2r11}
 In $\mathbf{ZF}$, a Cantor completely metrizable space (even a compact metrizable space) which has a cuf base may fail to be $\mathcal{K}$-Loeb and it may contain an infinite Dedekind-finite open subspace. Indeed, assuming that $\mathbf{CAC}_{fin}$ fails, we can fix a family $\mathcal{A}=\{A_n: n\in\omega\}$ of pairwise disjoint non-empty finite sets such that $\mathcal{A}$ does not have a partial choice function. Let $X=\bigcup\mathcal{A}$ and $\mathbf{X}=\langle X, \mathcal{P}(X)\rangle$. The one-point Hausdorff compactification $\alpha\mathbf{X}$ of $\mathbf{X}$ has a cuf base, is Cantor completely metrizable but not $\mathcal{K}$-Loeb. The set $X$ is Dedekind-finite  and open in $\alpha\mathbf{X}$, so $X$ is of types $F_{\sigma}$ and  $G_{\delta}$ in $\alpha\mathbf{X}$. 
\end{remark}

We do not know a satisfactory answer to the following question:

\begin{question}
\label{s2q12}
Is it true in $\mathbf{ZF}$ that every Cantor completely metrizable $\mathcal{K}$-Loeb space with a cuf base is Loeb?
\end{question}

To give an extended partial answer to Question \ref{s2q12}, we need to apply scattered spaces. Let us recall that a topological space $\mathbf{X}$ is called \emph{scattered} (or, equivalently, \emph{dispersed}) if no non-empty subspace of $\mathbf{X}$ is dense-in-itself.  The newest results on scattered spaces in $\mathbf{ZF}$ are included in \cite{ktw2}. 

For a topological space $\mathbf{X}$, we denote by $\Iso(X)$ the set of all isolated points of $\mathbf{X}$ and, for an ordinal $\alpha$, $X^{(\alpha)}$ stands for the $\alpha$-th Cantor-Bendixson derivative of $\mathbf{X}$.  The least ordinal $\alpha$ such that $X^{(\alpha+1)}=X^{(\alpha)}$ is denoted by $|\mathbf{X}|_{CB}$ and called the \emph{Cantor-Bendixson rank} of $\mathbf{X}$. If $\alpha=|\mathbf{X}|_{CB}$,  then $X=\bigcup_{\gamma\in\alpha}\Iso(\mathbf{X}^{(\gamma)})\cup X^{(\alpha)}$. 

\begin{theorem}
\label{s2t13}
$[\mathbf{ZF}]$
\begin{enumerate}
\item[(i)]  Every compact $\mathcal{K}$-Loeb space is Loeb.

\item[(ii)]  Every countably compact $\mathcal{K}$-Loeb space with a cuf base is compact, so also
Loeb.

\item[(iii)] Every scattered, $\mathcal{K}$-Loeb  $T_1$-space
with a cuf base is well-orderable, so also Loeb.

\item[(iv)] For every dense-in-itself $\mathcal{K}$-Loeb $T_3$-space $\mathbf{X}=\langle X, \tau\rangle$ such that the family $\mathcal{C}=\{O\setminus\{x\}: x\in O\in \tau\}$ contains a cuf base of $\mathbf{X}$, it holds that $\mathbf{X}$ is separable and second-countable, so also Loeb if $\mathbf{X}$ is Cantor completely metrizable.

\item[(v)]  Every Cantor completely metrizable, $\mathcal{K}$-Loeb space such that
each of its bases contains a cuf base is Loeb. 

\end{enumerate}
\end{theorem}

\begin{proof} Since closed sets in compact spaces are compact, it is obvious that (i) holds.  To prove (ii)-(iv), we assume that $\mathbf{X}=\langle X, \tau\rangle$ is a $\mathcal{K}$-Loeb $T_1$-space which has cuf base. We fix a base $\mathcal{B}=\bigcup \{\mathcal{B}_{i}:i\in \omega\}$ of $\mathbf{X}$ such that, for every $i\in\omega$, the family $\mathcal{B}_i$ is finite. We fix a choice function $h$ of $\mathcal{K}^{\ast}(\mathbf{X})$.\medskip

(ii) Suppose that $\mathbf{X}$ is countably compact. In view of (i), to prove (ii), it suffices to show that $\mathbf{X}$ is compact. To
this end, we fix an open cover $\mathcal{U}\subseteq \mathcal{B}$ of $\mathbf{X}
$. For every $i\in \omega$, let $V_{i}=\bigcup (\mathcal{B}_{i}\cap 
\mathcal{U})$. Clearly, $\mathcal{V}=\{V_{i}:i\in \omega\}$ is an open
cover of $\mathbf{X}$. Since $\mathbf{X}$ is countably compact, there exists $n\in\omega$ such that $X=\bigcup_{i\in n+1}V_i$. Then $\mathcal{W}=\bigcup \{\mathcal{B
}_i\cap \mathcal{U}:i\in n+1\}$ is a finite subcover of $\mathcal{U}$, so $\mathbf{X}$ is compact as required.\medskip

(iii) Now, suppose that $\mathbf{X}$ is scattered. We are going to define (effectively in $\mathbf{ZF}$) a choice function $f$ of $\mathcal{P}(X)\setminus\{\emptyset\}$. To this aim, we consider any non-empty subset $A$ of $X$. For every $n\in\omega$, let $A_n=\{x\in A: (\exists U\in\mathcal{B}_n) (U\cap A=\{x\})\}$. We notice that $\Iso(A)=\bigcup_{n\in\omega}A_n$. Since the space $\mathbf{X}$ is scattered, the set $\Iso(A)$ is non-empty. Then the set $N(A)=\{n\in\omega: A_n\neq\emptyset\}$ is non-empty, so we can define $n(A)=\min N(A)$. The set $A_{n(A)}$ is a non-empty finite subset of $A$. Hence, we can define $f(A)=h(A_{n(A)})$. This shows that $\mathbf{AC}(X)$ holds, so $X$ is well-orderable by Proposition \ref{s1:WOT}.\medskip

(iv) For the proof of (iv), we assume that $\mathbf{X}$ is a dense-in-itself $T_3$-space, and the family $\mathcal{C}$ contains a cuf a base for $\mathbf{X}$. We may assume that $\mathcal{B}\subseteq \mathcal{C}$. It follows from Theorem \ref{s1t14} that $\mathbf{X}$ is metrizable. For every $i\in \omega$, let 
$$A_{i}=\{x\in X: (\exists O\in \tau)(x\in O\wedge O\setminus\{x\}\in \mathcal{B}_i)\}$$
\noindent and let $A=\bigcup_{i\in\omega}A_i$. Since $\mathbf{X}$ is a $\mathcal{K}$-Loeb $T_1$-space and, for every $i\in\omega$, the set $A_i$ is finite, it follows from Proposition \ref{s2p2} that $A$ is countable. To show that $A$ is dense in $\mathbf{X}$, let us consider any non-empty set $U\in\tau$ and any $x\in U$. There exist $i\in\omega$, $O\in\tau$ and $y\in O$, such that $x\in O\setminus\{y\}\in \mathcal{B}_i$ and $\cl_{\mathbf{X}}(O\setminus \{y\})\subseteq U$. Then $y\in U$ because $\mathbf{X}$ is dense-in-itself. Furthermore, $y\in A$. This proves that $A$ is dense in $\mathbf{X}$. In consequence,  $\mathbf{X}$ is a separable metrizable space. Hence $\mathbf{X}$ is second-countable. By Theorem \ref{s1t12}, if $\mathbf{X}$ is Cantor completely metrizable, then $\mathbf{X}$ is Loeb.\medskip

(v) Suppose that  $\mathbf{Y}=\langle Y, \tau_Y\rangle$ is a Cantor completely metrizable $\mathcal{K}$-Loeb, non-scattered space whose every base contains a cuf base. Let $\alpha_Y$ be the Cantor-Bendixson rank of $\mathbf{Y}$. Let $X=Y^{(\alpha_Y)}$.  Then the closed subspace $\mathbf{X}$ of $\mathbf{Y}$ is a dense-in-itself, Cantor completely metrizable $\mathcal{K}$-Loeb space. Let $\mathcal{B}_X$ be a base of $\mathbf{X}$. Then $\mathcal{B}_Y=\{U\in\tau_Y: U\cap X=\emptyset\}\cup\{U\in\tau: U\cap X\in\mathcal{B}_X\}$ is a base of $\mathbf{Y}$. Since $\mathcal{B}_Y$ contains a cuf base of $\mathbf{Y}$, it is easily seen that $\mathcal{B}_X$ contains a cuf base of $\mathbf{X}$. Therefore, it follows from (iv) that $\mathbf{X}$ is second-countable and Loeb. In much the same way, as in the proof of Proposition \ref{s2p2}, we can show that there exists a family $\{\leq_{\gamma}: \gamma\in\alpha_Y\}$ such that, for every $\gamma\in\alpha_Y$, $\leq_{\gamma}$ is a well-ordering on $\Iso(Y^{{(\gamma)}})$. This implies that $Y\setminus X$ is well-orderable. In consequence, since $\mathbf{X}$ is Loeb and $Y\setminus X$ is well-orderable, the space $\mathbf{Y}$ is Loeb. 
\end{proof}

Every Loeb $T_2$-space is $\mathcal{K}$-Loeb in $\mathbf{ZF}$; however, the following proposition shows that, in $\mathbf{ZF}$, a Loeb $T_1$-space may fail to be $\mathcal{K}$-Loeb.

\begin{proposition}
\label{s2p14}
\begin{enumerate}
\item[(a)] $[\mathbf{ZF}]$ Let $X$ be a set and let $\tau_{cf}$ be the co-finite topology on $X$. Then $\mathbf{X}=\langle X, \tau_{cf}\rangle$ is a $T_1$-space for which the following equivalences hold:
\begin{enumerate}
\item[(i)] $\mathbf{AC}(X)\leftrightarrow \mathbf{KL}(\mathbf{X})$;
\item[(ii)] $\mathbf{AC}_{fin}(X)\leftrightarrow \mathbf{L}(\mathbf{X})$.
\end{enumerate}
\item[(b)] In Cohen's Original Model $\mathcal{M}1$ in \cite{hr}, there exists a Loeb $T_1$-space which is not $\mathcal{K}$-Loeb. 
\end{enumerate}
\end{proposition}
\begin{proof}
(a)  Since every subset of $\mathbf{X}$ is compact, it is obvious that (i) holds. Since $\mathcal{C}l(\mathbf{X})=[X]^{<\omega}\cup\{X\}$, it follows that (ii) holds.\medskip

(b) In Cohen's Original Model $\mathcal{M}1$, $\mathbf{AC}_{fin}$ holds and $\mathbf{AC}$ fails. This implies that, in $\mathcal{M}1$, we can fix a set $X$ for which $\mathbf{AC}_{fin}(X)\wedge\neg\mathbf{AC}(X)$ is true. For instance, in $\mathcal{M}1$, the set $D$ of all added Cohen reals is an infinite Dedekind-finite subset of $\mathbb{R}$ (see \cite[pp. 146--147]{hr}), so $D$ can serve as $X$. It follows from (a) that, in $\mathcal{M}1$, $\langle X, \tau_{cf}\rangle$ is a Loeb $T_1$-space which is not $\mathcal{K}$-Loeb. 

\end{proof}

\begin{remark}
\label{s2r10}
(i) Suppose that $\mathcal{M}$ is a model of $\mathbf{ZF}$ in which $\mathbf{AC}_{fin}$ is true and $\mathbf{CAC}$ is false. For instance, Cohen's Original Model $\mathcal{M}1$ in \cite{hr} can be taken as $\mathcal{M}$.  Then, in $\mathcal{M}$, there exists a Cantor completely metrizable $\mathcal{K}$-Loeb space which is not Loeb. Indeed, we can fix in $\mathcal{M}$ a denumerable family $\mathcal{A}$ of infinite sets without a choice function in $\mathcal{M}$. Then, in $\mathcal{M}$, the discrete space $\mathbf{X}$ with the underlying set $X=\bigcup\mathcal{A}$ is Cantor completely metrizable, $\mathcal{K}$-Loeb but not Loeb. The space $\mathbf{X}$ does not have a cuf base in $\mathcal{M}$.

(ii) Let $\mathbf{X}$ be a locally compact Hausdorff, second-countable space, and let $\mathbf{P}$ be a $G_{\delta}$-subspace of $\mathbf{X}$. Since $\mathbf{P}$ is second-countable and, by Corollary \ref{s2c8}, also Cantor completely metrizable, that $\mathbf{P}$ is Loeb follows from Theorem \ref{s1t12}. This is an alternative proof of Theorem \ref{s2t6} limited to second-countable, locally compact Hausdorff spaces.

(iii) If $\mathbf{X}$ is a Cantor completely metrizable second-countable space, one can modify our proof of Theorem \ref{s2t6} to show without applying Theorem \ref{s2t7} that, in $\mathbf{ZF}$, every $G_{\delta}$-subspace of $\mathbf{X}$ is Loeb.  
\end{remark}

\section{Equivalents of $\mathbf{S}(\mathbf{X})$ for a Cantor completely metrizable second-countable $\mathbf{X}$}
\label{s3}

The main aim of this section is to give several new necessary and sufficient conditions for a Cantor completely metrizable second-countable space to be sequential. Let us start with the following simple proposition which follows directly from Theorems \ref{s2t3}, \ref{s2t9} and Proposition \ref{s1p19}.

\begin{proposition}
\label{s3p1}
$[\mathbf{ZF}]$ Let $\mathbf{X}$ be a Cantor completely metrizable second-countable space and let $A$ be a sequentially closed set in $\mathbf{X}$. If $A$ is either of type $G_{\delta}$ or of type $F_{\sigma}$ in $\mathbf{X}$, then $A$ is closed in $\mathbf{X}$.
\end{proposition}

The following corollaries can be deduced from Proposition \ref{s3p1} and Theorem \ref{s1t12}:
\begin{corollary}
\label{s3c2}
$[\mathbf{ZF}]$  Let $A$ be a sequentially closed subset of $\mathbb{P}=\mathbb{R}\setminus\mathbb{Q}$. If $A$ is of type $G_{\delta}$ in $\mathbb{P}$ or of type $F_{\sigma}$ in $\mathbb{P}$, then $A$ is closed in $\mathbb{P}$. 
\end{corollary}

\begin{corollary}
\label{s3c3}
$[\mathbf{ZF}]$ Let $\mathbf{X}$ be a second-countable, Cantor completely metrizable space. Then $\mathcal{G}^{\ast}_{\delta}(\mathbf{X})\cap\mathcal{SC}(\mathbf{X})=\mathcal{F}^{\ast}_{\sigma}(\mathbf{X})\cap\mathcal{SC}(\mathbf{X})=Cl^{\ast}(\mathbf{X})$. In consequence, the families $\mathcal{G}^{\ast}_{\delta}(\mathbf{X})\cap\mathcal{SC}(\mathbf{X})$ and $\mathcal{F}^{\ast}_{\sigma}(\mathbf{X})\cap\mathcal{SC}(\mathbf{X})$ have choice functions.
\end{corollary}

\begin{theorem}
\label{s3t4}
$[\mathbf{ZF}]$ Let $\mathbf{X}$ be a second-countable, metrizable Loeb space. Then the following conditions are all equivalent:
\begin{enumerate}
\item[(i)]  $\mathbf{S}(\mathbf{X})$; 
\item[(ii)] every sequentially closed subspace of $\mathbf{X}$ is separable;
\item[(iii)] $\mathbf{SL}(\mathbf{X})$;
\item[(iv)] $\mathbf{SCF}_{\sigma}(\mathbf{X})$;
\item[(v)] $\mathbf{SCDF}_{\sigma}(\mathbf{X})$;
\item[(vi)] $\mathbf{SCDWOU}(\mathbf{X})$;
\item[(vii)] for every $F_{\sigma}$-subspace $\mathbf{P}$ of $\mathbf{X}$, $\mathbf{S}(\mathbf{P})$ holds. 
\end{enumerate}
In particular, for every second-countable, Cantor completely metrizable space $\mathbf{X}$, conditions (i)-(vii) are all equivalent.
\end{theorem}

\begin{proof}
Let $f$ be a Loeb function of $\mathbf{X}$ and let $\mathcal{B}=\{U_n: n\in\omega\}$ be a base of $\mathbf{X}$.

That (i) implies (ii) follows from Theorem \ref{s2t3}. By Proposition \ref{s1p19}, (ii) implies (i).  Assuming (i), we obtain that $\mathcal{SC}^{\ast}(\mathbf{X})=Cl^{\ast}(\mathbf{X})$; thus, since $\mathbf{X}$ is Loeb, the family $\mathcal{SC}^{\ast}(\mathbf{X})$ has a choice function. Hence (i) implies (iii). \medskip

To show that (iii) implies (ii), we assume (iii) and fix a non-empty sequentially closed subset $A$ of $\mathbf{X}$. Let $N=\{n\in\omega: A\cap U_n\neq\emptyset\}$. We notice that, for every $n\in N$, $A\cap\cl_{\mathbf{X}}(U_n)\in\mathcal{SC}^{\ast}(\mathbf{X})$. Let $g$ be a choice function of $\mathcal{SC}^{\ast}(\mathbf{X})$. Then the set $D(A,g)=\{ g(A\cap\cl_{\mathbf{X}}(U_n)): n\in N\}$ is a countable dense subset of the subspace $\mathbf{A}$ of $\mathbf{X}$. Hence (iii) implies (ii). This completes the proof that conditions (i)--(iii) are all equivalent.\medskip

It is trivial that (i)$\rightarrow$(iv)$\rightarrow$(v).  Now, assume (v) and fix $A\in\mathcal{SC}^{\ast}(\mathbf{X})$. By (v), there exists a set $F\subseteq A$ such that $F$ is of type $F_{\sigma}$ in $\mathbf{X}$ and $F$ is dense in the subspace $\mathbf{A}$ of $\mathbf{X}$.  Let $F=\bigcup_{i\in\omega}F_i$. For every $i\in\omega$, let $N_i=\{n\in\omega: F_i\cap U_n\neq\emptyset\}$ and $D_i=\{f(F_i\cap\cl_{\mathbf{X}}(U_n)): n\in N_i\}$. For $i\in\omega$, let $h_i: N_i\to D_i$ be defined as follows:  for every $n\in N_i$, $h_i(n)=f(F_i\cap\cl_{\mathbf{X}}(U_n))$. Then, given $i\in\omega$, we have a surjection $h_i$ of the subset $N_i$ of $\omega$ onto the set $D_i$, so we can effectively define a well-ordering $\leq_i$ on $D_i$. The set $D=\bigcup_{i\in\omega}D_i$ is a countable union of well-orderable sets and $D$ is dense in $A$. Hence (v) implies (vi). Let us show that (vi) implies (ii).\medskip

Assume (vi) and fix $A\in\mathcal{SC}^{\ast}(\mathbf{X})$. By (vi), there exist a von Neumann ordinal number $\kappa$ and a family $\{ D_{\alpha}: \alpha\in\kappa\}$ such that, for every $\alpha\in\kappa$, $D_{\alpha}$ is a well-orderable subset of $A$, and the set $D=\bigcup_{\alpha\in\kappa}D_{\alpha}$ is dense in the subspace $\mathbf{A}$ of $\mathbf{X}$. We fix $\alpha\in\kappa$ and put $A_\alpha=A\cap\cl_{\mathbf{X}}(D_{\alpha})$. Suppose that $x\in\cl_{\mathbf{X}}(A_{\alpha})$. Then $x\in\cl_{\mathbf{X}}(D_{\alpha})$. Since $D_{\alpha}$ is well-orderable, we can define a sequence $(x_n)_{n\in\omega}$ of members of $D_{\alpha}$ such that $\lim\limits_{n\to+\infty}x_n=x$  in $\mathbf{X}$. Since $A$ is sequentially closed in $\mathbf{X}$, we deduce that $x\in A$. This shows that the set $A_{\alpha}$ is closed in $\mathbf{X}$. 

Let $M=\{n\in\omega: A\cap U_n\neq\emptyset\}$. For every $n\in M$, let $C_n=\{\alpha\in\kappa: U_n\cap A_{\alpha}\neq\emptyset\}$. That the set $D$ is dense in $\mathbf{A}$ implies that, for every $n\in M$, the set $C_n$ is non-empty, so we can define $\alpha_n$ as the first ordinal number in $C_n$. The set $\{ f(A_{\alpha_n}\cap\cl_{\mathbf{X}}(U_n)): n\in M\}$ is countable and dense in $\mathbf{A}$. Hence (vi) implies (ii). In consequence, conditions (i)--(vi) are all equivalent. Furthermore, it is obvious that (vii) implies (i).\medskip

Assume (i). Let $\{H_n: n\in\omega\}$ be a family of closed subsets of $\mathbf{X}$ and let $P=\bigcup_{n\in\omega}H_n$. Let  $A$ be a sequentially closed set in the subspace $\mathbf{P}$ of $\mathbf{X}$. For every $n\in\omega$, the set $H_n\cap A$ is sequentially closed in $\mathbf{X}$, so it is closed in $\mathbf{X}$. This implies that $A=\bigcup_{n\in\omega}(H_n\cap A)$ is of type $F_{\sigma}$ in $\mathbf{X}$. By Theorem \ref{s2t3}, the subspace $\mathbf{A}$ of $\mathbf{X}$ is separable. This, together with Proposition \ref{s1p19}, implies that $A$ is closed in $\mathbf{P}$. Hence (i) implies (vii). In consequence, (i)-(vii) are all equivalent. Thus, to conclude the proof concerning the case when $\mathbf{X}$ is Cantor completely metrizable, it suffices to apply Theorem \ref{s1t12}.
\end{proof}

\begin{proposition}
\label{s3p5}
$[\mathbf{ZF}]$ Let $\mathbf{X}$ be a first-countable Loeb $T_3$-space. Then the following conditions are all equivalent:
\begin{enumerate}
\item[(i)] $\mathcal{S}(\mathbf{X})$;
\item[(ii)] for every sequentially closed subset $A$ of $\mathbf{X}$ and every
point $x\in \cl_{\mathbf{X}}(A)$, there exists a family $\{F_n: n\in\omega\}$ of closed sets of $\mathbf{X}$ such that $x\in\cl_{\mathbf{X}}(\bigcup_{n\in\omega}F_n)$ and, for every $n\in\omega$, $F_n\subseteq A$;
\item[(iii)] for every sequentially closed subset $A$ of $\mathbf{X}$ and every
point $x\in \cl_{\mathbf{X}}(A)$, there exists a family $\{F_n: n\in\omega\}$ of well-orderable subsets of $A$ such that $x\in\cl_{\mathbf{X}}(\bigcup_{n\in\omega}F_n)$.
\end{enumerate}
In particular, for every Cantor completely metrizable, second-countable space $\mathbf{X}$,  conditions (i)-(iii) are all equivalent.
\end{proposition}

\begin{proof} Let $A\in\mathcal{SC}^{\ast}(\mathbf{X})$ and let $x\in\cl_{\mathbf{X}}(A)$. Let $\mathcal{B}(x)=\{U_i: i\in\omega\}$ be a base of neighborhoods of $x$ in $\mathbf{X}$ such that, for every $i\in\omega$, $\cl_{\mathbf{X}}(U_{i+1})\subseteq U_{i}$. Let $f$ be a Loeb function of $\mathbf{X}$. \medskip

If (i) holds, then $x\in A$, so we can put $F_n=\{x\}$ for every $n\in\omega$. Hence (i) implies both (ii) and (iii).\medskip

To show that (ii) implies (i), assuming (ii), we fix a family $\{F_n: n\in\omega\}$ of closed sets of $\mathbf{X}$ such that $x\in\cl_{\mathbf{X}}(\bigcup_{n\in\omega}F_n)$ and, for every $n\in\omega$, $F_n\subseteq A$. For $i\in\omega$, let $n_i=\min\{n\in\omega: U_i\cap F_n\neq\emptyset\}$ and let $x_i=f(F_{n_i}\cap\cl_{\mathbf{X}}(U_i))$. Then $(x_i)_{i\in\omega}$ is a sequence of points of $A$ which converges in $\mathbf{X}$ to $x$. Since $A$ is sequentially closed, $x\in A$. Hence (ii) implies (i). \medskip

Now, we assume (iii) and fix a family $\{C_n: n\in\omega\}$ of well-orderable subsets of $A$ such that $x\in\cl_{\mathbf{X}}(\bigcup_{n\in\omega}C_n)$.  Since $A\in\mathcal{SC}(\mathbf{X})$ and, for every $n\in\omega$, the set $C_n$ is a well-orderable subset of $A$, arguing in much the same way, as in the proof of Proposition \ref{s1p19}, we deduce that, for every $n\in\omega$, $\cl_{\mathbf{X}}(C_n)\subseteq A$. For every $n\in\omega$, we put $F_n=\cl_{\mathbf{X}}(C_n)$ to see that (ii) is satisfied. Hence (iii) implies (ii) and, in consequence, conditions (i)-(iii) are all equivalent.  We conclude the proof by applying Theorem \ref{s1t12}.
\end{proof}

\section{Forms (9)--(17) of Definition \ref{s1d8}}
\label{s4}

In this section, we are concerned mainly with forms from Definitions \ref{s1d7} and \ref{s1d8} for metrizable spaces; however, if possible, we generalize our results to topological spaces from a wider class than the class of metrizable ones.

\begin{theorem}
\label{s4t1}
$[\mathbf{ZF}]$ Let $\mathbf{X}$ be a first-countable Hausdorff space such that $\mathcal{C}l(\mathbf{X})\subseteq\mathcal{G}_{\delta}(\mathbf{X})$. Then the following conditions are satisfied:
\begin{enumerate}
\item[(i)]  $\mathbf{IDI}(\mathbf{X})\rightarrow (\mathbf{\Pi G}_{\delta}(\mathbf{X})\wedge \mathbf{\Pi F}_{\sigma}(\mathbf{X})\wedge\mathbf{ISCIC}(\mathbf{X}))$;
\item[(ii)] if $\mathbf{X}$ is Cantor completely metrizable and second-countable, then 
$$(\mathbf{\Pi G}_{\delta}(\mathbf{X})\vee \mathbf{\Pi F}_{\sigma}(\mathbf{X})\vee\mathbf{ISCIC}(\mathbf{X}))\rightarrow \mathbf{IDI}(\mathbf{X}).$$
\end{enumerate}
\end{theorem}
\begin{proof}
Let $A$ be an infinite subset of $\mathbf{X}$.

(i) Assume $\mathbf{IDI}(\mathbf{X})$. Then $A$ is Dedekind-infinite, so, there exists an injection $\psi:\omega\to A$. Let $F=\psi[\omega]$. Then $F\in\mathcal{F}_{\sigma}(\mathbf{X})$ and $F\subseteq A$. This shows that $\mathbf{\Pi F}_{\sigma}(\mathbf{X})$ holds.  We are going to show that $F$ contains an infinite $G_{\delta}$-subset of $\mathbf{X}$. 

 Suppose  that the subspace  $\mathbf{F}$ of $\mathbf{X}$ is not discrete. Let $x_0\in F$ be an accumulation point of $F$ in $\mathbf{X}$. Then, since $X$ is first-countable and $F$ is denumarable, we can define a sequence $(x_n)_{n\in\mathbb{N}}$ of points of $F\setminus\{x_0\}$ which converges in $\mathbf{X}$ to the point $x_0$. The set $B=\{x_n: n\in\omega\}$ is closed in the space $\mathbf{X}$, so $B$  is of type $G_{\delta}$ in $\mathbf{X}$. Clearly, $B$ is an infinite subset of $A$.

Suppose that $\mathbf{F}$ is discrete. Let $D=\cl_{\mathbf{X}}(F)$. Then $F$ is open in the subspace $\mathbf{D}$ of $\mathbf{X}$. Since  $D$ is of type $G_{\delta}$ in $\mathbf{X}$ and $F$ is open in $\mathbf{D}$, $F$ is of type $G_{\delta}$ in $\mathbf{X}$. Hence $\mathbf{\Pi G}_{\delta}(\mathbf{X})$ holds.

Suppose that $A\in\mathcal{SC}(\mathbf{X})$. If $F\notin\mathcal{C}l(\mathbf{X})$, there exist $y_0\in\cl_{\mathbf{X}}(F)\setminus F$ and a sequence $(y_n)_{n\in\mathbb{N}}$ of points of $F$ which converges in $\mathbf{X}$ to $y_0$. Then $y_0\in A$ because $A$ is sequentially closed in $\mathbf{X}$. Therefore, the set $H=\{y_n: n\in\omega\}$ is an infinite subset of $A$ such that $H$ is closed in $\mathbf{X}$. Hence $\mathbf{ISCIC}(\mathbf{X})$ also holds. \medskip

(ii) Now, assume that $\mathbf{X}$ is a Cantor completely metrizable, second-countable space. First, suppose that $\mathbf{\Pi G}_{\delta}(\mathbf{X})\vee\mathbf{\Pi F}_{\sigma}(\mathbf{X})$ holds. Then there exists an infinite subset $E$ of $A$ such that $E\in\mathcal{F}_{\sigma}(\mathbf{X})\cup\mathcal{G}_{\delta}(\mathbf{X})$. It follows from Corollary \ref{s2c10} that $E$ is Dedekind-infinite, so $A$ is Dedekind-infinite, too. \medskip

Finally, to show that $\mathbf{ISCIC}(\mathbf{X})$ implies $\mathbf{IDI}(\mathbf{X})$, suppose that $A$ is Dedekind-finite. Then $A$ is sequentially closed in $\mathbf{X}$. Let $C$ be an infinite subset of $A$. Then $C$ is Dedekind-finite, so, by Corollary \ref{s2c10}, $C$ is not closed in $\mathbf{X}$. Hence, if $\mathbf{ISCIC}(\mathbf{X})$ holds, then $A$ is Dedekind-infinite.
\end{proof}

\begin{corollary}
\label{s4c2}
$[\mathbf{ZF}]$ For every second-countable,  Cantor completely metrizable space $\mathbf{X}$, the statements $\mathbf{IDI}(\mathbf{X})$, $\mathbf{\Pi G}_{\delta}(\mathbf{X})$, $\mathbf{\Pi F}_{\delta}(\mathbf{X})$ and $\mathbf{ISCIC}(\mathbf{X})$ are all equivalent.
\end{corollary}

\begin{proposition}
\label{s4p3}
$[\mathbf{ZF}]$
\begin{enumerate}
\item[(i)] For every topological space $\mathbf{X}$, the following implication is true:  
$$\mathbf{UIC}(\mathbf{X})\rightarrow\mathbf{UISC}(\mathbf{X}).$$
\item[(ii)] $\mathbf{AC}\leftrightarrow\mathbf{F_{\sigma}L}_M\leftrightarrow\mathbf{G_{\delta}L}_M$.
\item[(iii)] $\mathbf{CAC}\leftrightarrow\mathbf{CF_{\sigma}L}_M\leftrightarrow\mathbf{CG_{\delta}L}_M$.
\end{enumerate}
\end{proposition}

\begin{proof}
It is obvious that (i) holds,  $\mathbf{AC}$ implies both $\mathbf{F_{\sigma}L}_M$ and $\mathbf{G_{\delta}L}_M$, while $\mathbf{CAC}$ implies both $\mathbf{CF_{\sigma}L}_M$ and $\mathbf{CG_{\delta}L}_M$.
For the proof of (ii) and (iii), let us consider any non-empty family $\mathcal{A}$ of non-empty sets, put $X=\bigcup\mathcal{A}$ and $\mathbf{X}=\langle X, \mathcal{P}(X)\rangle$. Then $\mathbf{X}$ is metrizable and $\mathcal{A}\subseteq\mathcal{F}^{\ast}_{\sigma}(\mathbf{X})\cap\mathcal{G}^{\ast}_{\delta}(\mathbf{X})$. Hence $\mathbf{F_{\sigma}L}_M\vee\mathbf{G_{\delta}L}_M$ implies that $\mathcal{A}$ has a choice function, while $\mathbf{CF_{\sigma}L}_M\vee\mathbf{CG_{\delta}L}_M$ implies that if $\mathcal{A}$ is countable, then $\mathcal{A}$ has a choice function.
\end{proof}

\begin{theorem}
\label{s4t4}
$[\mathbf{ZF}]$ Let $\mathbf{X}$ be a Hausdorff first-countable space. Then:
\begin{enumerate}
\item[(i)] $\mathbf{UIC}(\mathbf{X})\leftrightarrow (\mathbf{ISCIC}(\mathbf{X})\wedge\mathbf{UISC}(\mathbf{X}))$;
\item[(ii)] $(\mathbf{IDI}(\mathbf{X})\wedge\mathbf{UISC}(\mathbf{X}))\rightarrow \mathbf{UIC}(\mathbf{X})$;
\item[(iii)] if $\mathbf{X}$ is Cantor completely metrizable and second-countable, then 
  $$(\mathbf{IDI}(\mathbf{X})\wedge\mathbf{UISC}(\mathbf{X}))\leftrightarrow \mathbf{UIC}(\mathbf{X}).$$
\end{enumerate}
\end{theorem}
\begin{proof}
(i) To begin, assume $\mathbf{UIC}(\mathbf{X})$. Then $\mathbf{UISC}(\mathbf{X})$ holds (cf. Proposition \ref{s4p3}(i)). Let $A$ be an infinite sequentially closed subset of $\mathbf{X}$. If $A$ is Dedekind-finite, then $A$ is uncountable, so, by $\mathbf{UIC}(\mathbf{X})$,  $A$ contains an infinite closed subset of $\mathbf{X}$. Suppose that $A$ is Dedekind-infinite. Let $F$ be a denumerable subset of $A$. In much the same way, as in the last part of the proof of Theorem \ref{s4t1}(i), we can show that $A$ contains an infinite set $H$ such that $H$ is closed in $\mathbf{X}$. Hence $\mathbf{UIC}(\mathbf{X})\rightarrow (\mathbf{ISCIC}(\mathbf{X})\wedge\mathbf{UISC}(\mathbf{X}))$.

 Now, we assume that the conjunction $\mathbf{ISCIC}(\mathbf{X})\wedge\mathbf{UISC}(\mathbf{X})$ is true. Let $A$ be an uncountable subset of $X$. By $\mathbf{UISC}(\mathbf{X})$, there exists an infinite set $C\subseteq A$ such that $C$ is sequentially closed in $\mathbf{X}$. By  $\mathbf{ISCIC}(\mathbf{X})$, there exists an infinite set $D\subseteq C$ such that $D$ is closed in $\mathbf{X}$. This shows that $(\mathbf{ISCIC}(\mathbf{X})\wedge\mathbf{UISC}(\mathbf{X}))\rightarrow\mathbf{UIC}(\mathbf{X})$. Hence (i) holds.\medskip
 
 (ii) Assume that $\mathbf{IDI}(\mathbf{X})\wedge\mathbf{UISC}(\mathbf{X})$ is true.  Let $A$ be an uncountable subset of $\mathbf{X}$. By $\mathbf{UISC}(\mathbf{X})$, there exists an infinite set $C\in\mathcal{SC}(\mathbf{X})$ such that $C\subseteq A$. By $\mathbf{IDI}(\mathbf{X})$, there exists a denumerable subset $D$ of $C$. If $D$ is not closed in $\mathbf{X}$, in much the same way, as in the last part of the proof of Theorem \ref{s4t1}(i), we deduce that there exists a sequence $(y_n)_{n\in\mathbb{N}}$ of points of $D$ which converges in $\mathbf{X}$ to a point $y_0$  such that $y_0\notin D$. Since $C$ is sequentially closed, we have $y_0\in C$. Then $\{y_n: n\in\omega\}$ is an infinite subset of $A$ such that $\{y_n: n\in\omega\}\in\mathcal{C}l(\mathbf{X})$. Hence (ii) holds.\medskip
 
(iii) Let $\mathbf{X}$ be a Cantor completely metrizable, second-countable space for which $\mathbf{UIC}(\mathbf{X})$ is satisfied. Consider an arbitrary infinite subset $A$ of $X$. If $A$ is countable, then $A$ is Dedekind-infinite. Suppose that $A$ is uncountable. By $\mathbf{UIC}(\mathbf{X})$, we can fix an infinite subset $E$ of $A$ such that $E\in\mathcal{C}l(\mathbf{X})$. It follows from Corollary \ref{s2c10} that $E$ is Dedekind-infinite, so $A$ is Dedekind-infinite, too. Hence $\mathbf{UIC}(\mathbf{X})$ implies $\mathbf{IDI}(\mathbf{X})$. This, together with (ii) and Proposition \ref{s4p3}(i), completes the proof. 
\end{proof}

It is worth noticing that the following proposition holds:

\begin{proposition}
\label{s4p5}
$[\mathbf{ZF}]$ Let $\mathbf{X}$ be a second-countable Hausdorff space and let $A$ be an uncountable subset of $X$ such that the subspace  $\mathbf{A}$ of $\mathbf{X}$ is sequential. Then $A$ contains an infinite closed subset of $\mathbf{X}$.
\end{proposition}
\begin{proof}
Since $\mathbf{X}$ is second-countable, the space $\mathbf{A}$ is not discrete. Let $a_0\in A$ be an accumulation point of $\mathbf{A}$. Then the set $C=A\setminus\{a_0\}$ is not closed in $\mathbf{A}$. Since $\mathbf{A}$ is sequential, the set $C$ is not sequentially closed in $\mathbf{A}$. This implies that there exists a sequence $(a_n)_{n\in\mathbb{N}}$ of points of $C$ which converges in $\mathbf{A}$ to $a_0$. Then the set $D=\{a_n: n\in\omega\}$ is infinite, $D\subseteq A$ and $D$ is closed in $\mathbf{X}$. 
\end{proof}

\begin{corollary}
\label{s4c6}
$[\mathbf{ZF}]$ Let $\mathbf{X}$ is a second-countable Hausdorff space. Then:
$$\mathbf{CAC}(\mathbb{R})\rightarrow\mathbf{FU}(\mathbf{X})\rightarrow\mathbf{UIC}(\mathbf{X}).$$
\end{corollary}

\begin{proof}
The first implication follows from Theorem \ref{s1t16}. To prove that the second implication is true, let us assume $\mathbf{FU}(\mathbf{X})$ and consider any uncountable subset $A$ of $X$. By $\mathbf{FU}(\mathbf{X})$, the subspace $\mathbf{A}$ of $\mathbf{X}$ is sequential. Hence, by Proposition \ref{s4p5}, there exists an infinite set $C\in\mathcal{C}l(\mathbf{X})$ such that $C\subseteq A$.
\end{proof}

\begin{proposition} 
\label{s4p7}
$[\mathbf{ZF}]$ Let $\mathbf{X}=\langle X, \tau\rangle$ be a second-countable Hausdorff space. Then $\omega-\mathbf{CAC}(X)$ implies $\mathbf{UIC}(\mathbf{X})$.
\end{proposition}

\begin{proof}
Let $A$ be an uncountable subset of $X$. Let $\mathcal{B}=\{B_n: n\in\omega\}$ be a base of $\mathbf{X}$ and, for every $n\in\omega$, let $D_n=B_n\cap A$. Let  $M=\{n\in\omega: D_n\neq\emptyset\}$. Assuming $\omega-\mathbf{CAC}(Y)$, we can fix a family $\{C_n: n\in M\}$ of non-empty countable sets such that, for every $n\in M$, $C_n\subseteq D_n$. 

Suppose that there exists $n_0\in M$ such that the subspace $\mathbf{C}_{n_0}$ of $\mathbf{X}$ is not discrete. Then we can fix an accumulation point $x_0$ of $\mathbf{C}_{n_0}$. Since $\mathbf{X}$ is first-countable and $C_{n_0}$ is countable, there exists a sequence $(x_n)_{n\in\mathbb{N}}$ of points of $C_{n_0}\setminus\{x_0\}$ which converges in $\mathbf{X}$ to $x_0$. Then the set $E=\{x_n: n\in\omega\}$ is infinite, $E\in\mathcal{C}l(\mathbf{X})$ and $E\subseteq A$.  

Now, suppose that, for every $n\in M$, the subspace $\mathbf{C}_n$ of $\mathbf{X}$ is discrete. Then, for a fixed $n\in M$, let $k(n)=\min\{i\in\omega: |B_i\cap C_n|=1\}$ and let $y_{k(n)}$ be the unique element of $B_{k(n)}\cap C_n$. The set $Y=\{y_{k(n)}: n\in M\}$ is countable and dense in the subspace $\mathbf{A}$ of $\mathbf{X}$. Since $\mathbf{A}$ is an uncountable second-countable space, it is not discrete. Let $z$ be an accumulation point of $\mathbf{A}$. Then there exists a sequence of points of $Y\setminus\{z\}$ which converges in $\mathbf{X}$ to $z$. This implies that $A$ contains an infinite closed subset of $\mathbf{X}$.
\end{proof}

At this moment, we do not know satisfactory answers to the following questions:

\begin{question}
\label{s4q8}
\begin{enumerate}
\item[(i)] Does $\mathbf{IDI}(\mathbb{R})$ imply $\mathbf{UISC}(\mathbb{R})$ in $\mathbf{ZF}$?
\item[(ii)] Is there a model $\mathcal{M}$ of $\mathbf{ZF}+\mathbf{IDI}$ in which there exists a second-countable Cantor completely metrizable space $\mathbf{X}$ for which $\mathbf{UISC}(\mathbf{X})$ fails in $\mathcal{M}$?
\end{enumerate}
\end{question}

Let us have a deeper look at $\mathbf{F_{\sigma}L}(\mathbf{X})$ and $\mathbf{G_{\delta}L}(\mathbf{X})$.

\begin{theorem}
\label{s4t9}
$[\mathbf{ZF}]$ Let $\mathbf{X}$ be a $T_1$-space. Then:
\begin{enumerate}
\item[(i)] $\mathbf{G_{\delta}L}(\mathbf{X})$ implies that there exists a family $\{\langle A_{G}, \leq_{G}\rangle: G\in \mathcal{G}_{\delta}(\mathbf{X})\cap\mathcal{P}^{un}(X)\}$ such that, for every $G\in \mathcal{G}_{\delta}(\mathbf{X})\cap\mathcal{P}^{un}(X)$, $A_G\subseteq G$,$|A_{G}|=\aleph_1$, and $\leq_{G}$ is a well-ordering on $A_G$. In particular, if $\mathbf{X}$ is uncountable, then $\mathbf{G_{\delta}L}(\mathbf{X})$ implies that $\mathbf{X}$ contains an uncountable well-orderable subset.

\item[(ii)] $\mathbf{F_{\sigma}L}(\mathbf{X})$ implies that there exists a family $\{\langle A_{F}, \leq_{F}\rangle: F\in \mathcal{F}_{\sigma}(\mathbf{X})\cap\mathcal{P}^{inf}(X)\}$ such that, for every $F\in \mathcal{F}_{\sigma}(\mathbf{X})\cap\mathcal{P}^{inf}(X)$, $A_F\subseteq F$,$|A_{F}|=\aleph_0$, and $\leq_{F}$ is a well-ordering on $A_F$. In particular, $\mathbf{F_{\sigma}L}(\mathbf{X})$ implies that, if $F\in\mathcal{F}_{\sigma}(\mathbf{X})\cap\mathcal{P}^{inf}(X)$, then there is an effective way to define a denumerable, well-ordered subset of $F$.

\item[(iii)] $\mathbf{F_{\sigma}L}(\mathbf{X})$ implies that, for every well-orderable family $\mathcal{A}$ of countable subsets of $\mathbf{X}$, the union $\bigcup\mathcal{A}$ is well-orderable. 
\end{enumerate} 
\end{theorem} 

\begin{proof}
(i) Assuming $\mathbf{G_{\delta}L}(\mathbf{X})$, we fix a choice function $g$ of $\mathcal{G}_{\delta}^{\ast}(\mathbf{X})$. Let $G\in\mathcal{G}_{\delta}(\mathbf{X})\cap\mathcal{P}^{un}(X)$. By a trasitive induction on $\omega_1$, we can effectively define an injection $\psi_G:\omega_1\to G$ as follows. Let $\psi_G(0)=g(G)$. If $\gamma\in\omega_1$ is such that $\psi_G(\gamma)$ has been defined, we define $\psi_G(\gamma+1)=g(G\setminus\psi_G[\gamma+1])$. If $\alpha\in\omega_1$ is a limit ordinal such that, for every $\gamma\in\alpha$, $\psi_G(\gamma)$ has been defined, we put $\psi_G(\alpha)= g(G\setminus\psi_G[\alpha])$. We put $A_G=\psi[\omega_1]$ and, for $x,y\in A_G$, we write $x\leq_G y$ if and only if $\psi_G^{-1}(x)\subseteq \psi_G^{-1}(y)$. Hence (i) holds. \medskip

For the proof of (ii) and (iii), assuming $\mathbf{F_{\sigma}L}(\mathbf{X})$, we fix a choice function $f$ of $\mathcal{F}^{\ast}_{\sigma}(\mathbf{X})$.\medskip

(ii) Let $F\in\mathcal{F}_{\sigma}(\mathbf{X})\cap\mathcal{P}^{inf}(X)$. By an induction  on $\omega$, we can effectively define an injection $\phi_F:\omega\to F$ as follows. Let $\phi_F(0)=f(F)$. If $n\in\omega$ is such that $\phi_F(n)$ has been defined, we define $\phi_F(n+1)=f(F\setminus\phi_F[n+1])$.  We put $A_F=\phi[\omega]$ and, for $x,y\in A_F$, we write $x\leq_F y$ if and only if $\phi_F^{-1}(x)\subseteq \phi_F^{-1}(y)$. Hence (iii) holds.\medskip

(iii)  Let $\kappa$ be a von Neumann ordinal and let $\mathcal{A}=\{A_{\alpha}: \alpha\in\kappa\}$ be a family of non-empty countable subsets of $\mathbf{X}$. Let $\alpha\in\kappa$. For every $S\in\mathcal{P}(A_{\alpha})\setminus\{\emptyset\}$, we have $S\in\mathcal{F}^{\ast}_{\sigma}(\mathbf{X})$. Therefore, we can define a choice function $f_{\alpha}$ of $\mathcal{P}(A_{\alpha})\setminus\{\emptyset\}$ by putting, for every $S\in\mathcal{P}(A_{\alpha})\setminus\{\emptyset\}$, $f_{\alpha}(S)=f(S)$. Arguing in much the same way, as in the standard proof of the Well-Ordering Theorem (cf., e.g., \cite[the proof of Theorem 1.4]{her}), we can effectively define a well-ordering $\leq_{\alpha}$ on $A_{\alpha}$. This proves that there exists a family $\{\leq_{\alpha}: \alpha\in\kappa\}$ such that, for every $\alpha\in\kappa$, $\leq_{\alpha}$ is a well-ordering on $A_{\alpha}$. This implies that $\bigcup_{\alpha\in\kappa}A_{\alpha}$ is well-orderable. 
\end{proof}

\begin{theorem}
\label{s4t10}
$[\mathbf{ZF}]$ Let $\mathbf{X}$ be a $T_1$-space. Then:
\begin{enumerate}
\item[(i)] $(\mathbf{F_{\sigma}L}(\mathbf{X})\wedge\mathbf{KW}(X, \omega))\leftrightarrow\mathbf{AC}(X)$;
\item[(ii)] $(\mathbf{F_{\sigma}L}(\mathbf{X})\wedge\mathbf{CKW}(X, \omega))\rightarrow\mathbf{CAC}(X)$;
\item[(iii)] $(\mathbf{CF_{\sigma}L}(\mathbf{X})\wedge\mathbf{CKW}(X, \omega))\leftrightarrow\mathbf{CAC}(X)$.
\end{enumerate}
\end{theorem}
\begin{proof}
(i) It is obvious that $\mathbf{AC}(\mathbf{X})$ implies $\mathbf{F_{\sigma}L}(\mathbf{X})\wedge\mathbf{KW}(X, \omega)$. Assume that $\mathbf{F_{\sigma}L}(\mathbf{X})\wedge\mathbf{KW}(X, \omega)$ is true. Let $f$ be a choice function of $\mathcal{F}_{\sigma}^{\ast}(\mathbf{X})$. Let $h$ be a function defined on $\mathcal{P}(X)\setminus\{\emptyset\}$ such that, for every $A\in \mathcal{P}(X)\setminus\{\emptyset\}$, the set $h(A)$ is a non-empty countable subset of $A$. Then, for every $A\in\mathcal{P}(X)\setminus\{\emptyset\}$, we have $h(A)\in\mathcal{F}_{\sigma}^{\ast}(\mathbf{X})$, so we can put $\psi(A)=f( h(A))$ to define a choice function $\psi$ of $\mathcal{P}(X)\setminus\{\emptyset\}$. Hence (i) holds. Using similar arguments, one can prove (ii) and (iii).
\end{proof}

\section{On condensation points}
\label{s5}

We recall that, for an uncountable subset $A$ of a topological space $\mathbf{X}=\langle X, \tau\rangle$, a point $x\in X$ is called a \emph{condensation point} of $A$ in $\mathbf{X}$ if, for every $U\in\tau$ with $x\in U$, the set $A\cap U$ is uncountable. It is known that $\mathbf{CUC}(\mathbb{R})$ is equivalent to the sentence: ``Every uncountable subset $A$ of $\mathbb{R}$ has a condensation point $a\in A$.'' (Cf. \cite[Form 6A]{hr}.) 

\begin{theorem}
\label{s5t1}
$[\mathbf{ZF}]$
The following conditions (i)-(iii) are all equivalent:
\begin{enumerate}
\item[(i)] $\mathbf{CUC}$;
\item[(ii)] for every topological space $\mathbf{X}$ such that $\mathbf{X}$ has a cuf base, it holds that every uncountable subset $A$ of $X$ has a condensation point $x$ in $\mathbf{X}$ such that $x\in A$;
\item[(iii)] for every second-countable space $\mathbf{X}$, it holds that every uncountable subset $A$ of $X$ has a condensation point $x$ in $\mathbf{X}$ such that $x\in A$.
\end{enumerate}
Furthermore, $(iii)\rightarrow (iv)\rightarrow (v)$ where:
\begin{enumerate}
\item[(iv)] every uncountable subset $A$ of $\mathbb{R}$ has a condensation point $x$ in $\mathbb{R}$ such that $x\in A$;
\item[(v)] $\mathbf{CAC}_{\omega}(\mathbb{R})$.
\end{enumerate}
\end{theorem}
\begin{proof}
Let $\mathbf{X}$ be a topological space which has a cuf base. Let $\mathcal{B}=\bigcup_{n\in\omega}\mathcal{B}_n$  be a base of $\mathbf{X}$ such that, for every $n\in\omega$, the family $\mathcal{B}_n$ is finite. Suppose that $A$ is a subset of $\mathbf{X}$ such that no point of $A$ is a condensation point of $A$ in $\mathbf{X}$. Let $N(A)=\{n\in\omega: (\exists U\in \mathcal{B}_n) |U\cap A|\leq\aleph_0\}$. For every $n\in N(A)$, let $A_n=\bigcup\{U\cap A: U\in\mathcal{B}_n\wedge |U\cap A|\leq\aleph_0\}$  Then, for every $n\in N(A)$, the set $A_n$ is countable. Moreover,  $A=\bigcup_{n\in N(A)}A_n$. It follows from $\mathbf{CUC}$ that $A$ is countable.  Hence (i) implies (ii). The implications (ii)$\rightarrow $(iii)$\rightarrow$(iv) are obvious.

 Now, assuming that $\mathbf{CUC}$ fails, we fix a family $\{Y_n: n\in\omega\}$ of pairwise disjoint countable sets such that the set $Y=\bigcup_{n\in\omega}Y_n$ is uncountable. Let $\tau$ be the topology on $Y$ such that the family $\{Y_n: n\in\omega\}$ is a base of $\mathbf{Y}=\langle Y, \tau\rangle$. Then $\mathbf{Y}$ is second-countable but no point of $Y$ is a condensation point of $Y$ in $\mathbf{Y}$. Hence (iii) implies (i) and, in consequence (i)--(iii) are all equivalent.

Suppose that (v) is false. Then there exists a family $\mathcal{A}=\{A_n: n\in\mathbb{N}\}$ of countable subsets of $\mathbb{R}$ such that $\mathcal{A}$ does not have a choice function and, for every $n\in\mathbb{N}$, $A_n\subseteq (\frac{1}{n+1}, \frac{1}{n})$. Then the set $A=\bigcup_{n\in\mathbb{N}}A_n$ is uncountable and no point of $A$ is a condensation point of $A$ in $\mathbb{R}$. Hence (iv) implies (v).
\end{proof}

The proof that (i) implies (ii) in Theorem \ref{s5t1} shows that the following proposition holds:

\begin{proposition}
\label{s5p2}
$[\mathbf{ZF}]$ Let $A$ be a subset of a topological space $\mathbf{X}$ having a cuf base.  Suppose that no point of $A$ is a condensation point of $A$ in $\mathbf{X}$. Then $A$ is expressible as a countable union of countable sets. 
\end{proposition}

\begin{theorem}
\label{s5t3}
$[\mathbf{ZF}]$ The sentence ``For every metrizable space $\mathbf{X}$ such that $\mathbf{X}$ has a cuf base, it holds that every uncountable subset $A$ of $X$ has a condensation point $x$ in $\mathbf{X}$ such that $x\in A$''  implies $\mathbf{CAC}_{fin}$.
\end{theorem}
\begin{proof}
Assuming that $\mathbf{CAC}_{fin}$ is false, we fix a disjoint family $\mathcal{A}=\{A_n: n\in\omega\}$ of finite sets such that $\mathcal{A}$ does not have a choice function. Let $X=\bigcup_{n\in\omega}A_n$ and $\mathbf{X}=\langle X, \mathcal{P}(X)\rangle$. Then the space $\mathbf{X}$ is metrizable and has a cuf base. The set $X$ is uncountable and has no condensation point in $\mathbf{X}$.  
\end{proof}

Similarly to Proposition \ref{s4p5} and some other results of Section \ref{s4}, the following two propositions give sufficient conditions for a topological space $\mathbf{X}$ and an infinite subset of $X$ to contain an infinite closed set of $\mathbf{X}$. 

\begin{proposition}
\label{s5t4}
$[\mathbf{ZF}]$ Let $\mathbf{X}$ be a $T_3$-space which has a cuf base. Let $A$ be an uncountable sequentially closed subset of $\mathbf{X}$ such that no point of $A$ is a condensation point of $A$ in $\mathbf{X}$. Then $\mathbf{CAC}_{fin}(X)$ implies that $A$ contains an infinite closed set of $\mathbf{X}$.
\end{proposition}
\begin{proof}
By Proposition \ref{s5p2}, we can fix a family $\{A_n: n\in\omega\}$ of countable sets such that $A=\bigcup_{n\in\omega}A_n$. Since $A$ is uncountable, assuming $\mathbf{CAC}_{fin}(X)$, we deduce that there exists $n_0\in\omega$ such that the set $A_{n_0}$ is infinite. Suppose that $A_{n_0}$ is not closed in $\mathbf{X}$. Then there exists a point $x_0\in\cl_{\mathbf{X}}(A_{n_0})\setminus A_{n_0}$. In view of Theorem \ref{s1t14}, $\mathbf{X}$ is metrizable, so also first-countable. Hence, since  $A_{n_0}$ is countable, there exists a sequence $(x_n)_{n\in\mathbb{N}}$ of points of $A_{n_0}$ which converges in $\mathbf{X}$ to $x_0$. Since $A$ is sequentially closed in $\mathbf{X}$, $x_0\in A$. Then the set $C=\{x_n: n\in\omega\}$ is an infinite subset of $A$ such that $C\in\mathcal{C}l(\mathbf{X})$.
\end{proof}

Since every countable union of finite subsets of $\mathbb{R}$ is countable, arguing in much the same way, as in the proof of Proposition \ref{s5t4}, one can show that the following proposition holds:

\begin{proposition}
\label{s5p5} 
$[\mathbf{ZF}]$ If $A$ is an uncountable, sequentially closed
subset of $\mathbb{R}$ such that no point of $A$ is a condensation point of $A$ in $\mathbb{R}$, then there exists an infinite subset $C$ of $A$ such that $C$ is closed in $\mathbb{R}$.
\end{proposition}

\section{On $k$-spaces and Fr\'echet-Urysohn spaces}
\label{s6}

This section is devoted to the following statement due to  Arkhangel'skii: ``Every very $k$-space is Fr\'echet-Urysohn'' (see Theorem \ref{s1t:arkh}).  We show that, in $\mathbf{ZF}$, a very $k$-space may fail to be Fr\'echet-Urysohn and, among other problems, we search for sufficient conditions for a very $k$-space to be Fr\'echet-Urysohn.

We make use of the following idea several times in the sequel. This idea has appeared, for instance, in \cite{kerta}, \cite{kert}, \cite{kopsw}, \cite{kw1}, \cite{kw3} and \cite{ktw3}.

Suppose that  $\mathcal{A}=\{A_n: n\in\mathbb{N}\}$ is a disjoint family of non-empty sets, $A=\bigcup\mathcal{A}$ and $\infty\notin A$. Let $X=A\cup\{\infty\}$.  Suppose that $(\rho_n)_{n\in\mathbb{N}}$ is a sequence such that, for each $n\in\mathbb{N}$, $\rho_n$ is a metric on $A_n$. Let $d_n(x,y)=\min\{\rho_n(x,y), \frac{1}{n}\}$ for all $x,y\in A_n$. We define a function $d:X\times X\to\mathbb{R}$ as follows:

\[(\ast)\text{  } d(x,y)=\begin{cases} 0 &\text{if $x=y$;}\\
\max\{\frac{1}{n}, \frac{1}{m}\} &\text{if $x\in A_n ,y\in A_m$ and $n\neq m$;}\\
d_n(x,y) &\text{if  $x,y\in A_n$;}\\
\frac{1}{n} &\text{if $x\in A_n$ and $y=\infty$ or $x=\infty$ and $y\in A_n$.}\end{cases}
\]

\begin{proposition}
\label{s6p1}
The function $d$, defined by ($\ast$), has the following properties:
\begin{enumerate}
\item[(i)] $d$ is a metric on $X$ (cf. \cite{kw3}, \cite{ktw3});
\item[(ii)] if, for every $n\in\mathbb{N}$, the space $\langle A_n, \tau(\rho_n)\rangle$ is compact, then so is the space $\langle X, \tau(d)\rangle$ (cf. \cite{kw3}, \cite{ktw3});
\item[(iv)] if  $\mathcal{A}$ does not have a choice function, the space $\langle X, \tau(d)\rangle$ is not separable nor Loeb (cf. \cite{ktw3}, \cite{kerta});
\item[(v)] if $\mathcal{A}$ does not have a partial choice function, the space  $\langle X, \tau(d)\rangle$ is not sequential (cf. \cite{kw1});
\item[(vi)] if, for every $n\in\mathbb{N}$, the set $A_n$ is finite, then $\langle X, \tau(d)\rangle$ is a compact very $k$-space.
\end{enumerate}
\end{proposition}
\begin{proof}
Since only (vi) is new here, let us explain why (vi) holds. We assume that, for every $n\in\mathbb{N}$, the set $A_n$ is finite. By (ii), the space $\mathbf{X}=\langle X, \tau(d)\rangle$ is compact.  Let $\mathbf{P}$ be a subspace of $\mathbf{X}=\langle X, \tau(d)\rangle$.  Suppose that $D\subseteq P$ is such that, for every $K\in\mathcal{K}^{\ast}(\mathbf{P})$, $D\cap K\in\mathcal{K}(\mathbf{P})$. We need to show that $D$ is closed in the subspace $\mathbf{P}$ of $\mathbf{X}$. Suppose that $x_0\in\cl_{\mathbf{P}}(D)\setminus D$. Then $x_0=\infty$. The set $K_0=\{\infty\}\cup\bigcup_{n\in\mathbb{N}}(D\cap A_n)$ is compact in $\mathbf{X}$, $K_0\subseteq P$ but $K_0\cap D$ is not closed in $\mathbf{P}$. The contradiction obtained proves that $D\in\mathcal{C}l(\mathbf{P})$. Hence $\mathbf{X}$ is a very $k$-space. 
\end{proof}

\begin{theorem}
\label{s6t2}
$[\mathbf{ZF}]$ Let $\mathbf{X}$ be a first-countable, $\mathcal{K}$-Loeb  $T_3$-space. Then $\mathbf{X}$ is a very $k$-space if and only if $\mathbf{X}$ is Fr\'echet-Urysohn.
\end{theorem}
\begin{proof}
By Proposition \ref{s1p18}, every Fr\'echet-Urysohn Hausdorff space is a very $k$-space. Therefore, assuming that $\mathbf{X}$ is a very $k$-space, we prove $\mathbf{FU}(\mathbf{X})$.  Suppose that $A\subseteq X$ and  $x\in\cl_{\mathbf{X}}(A)\setminus A$. Let $P=A\cup\{x\}$. Since the subspace $\mathbf{P}$ of $\mathbf{X}$ is a $k$-space and $A$ is not closed in $\mathbf{P}$, there exists a compact set $K$ in $\mathbf{X}$ such that $K\subseteq P$ and $K\cap A$ is not closed in $\mathbf{P}$. Since $K\cup\{x\}$ is compact in $\mathbf{X}$ but $K\setminus\{x\}=K\cap A$ is not closed in $\mathbf{P}$,  we infer that $x\in\cl_{\mathbf{P}}(K\cap A)$. This implies that
 $x\in\cl_{\mathbf{X}}(K\setminus\{x\})$. Let $\mathcal{B}(x)=\{U_n: n\in\omega\}$ be a base of neighborhoods of $x$ in $\mathbf{X}$ such that, for every $n\in\omega$, $\cl_{\mathbf{X}}(U_{n+1})\subseteq U_n$. Let $f$ be a choice function of $\mathcal{K}^{\ast}(\mathbf{X})$. Let $M=\{\langle n, m\rangle\in\omega\times\omega: (\cl_{\mathbf{X}}(U_n)\setminus U_m)\cap K\neq\emptyset\}$. For every $\langle n, m\rangle\in M$, let $x_{n,m}=f((\cl_{\mathbf{X}}(U_n)\setminus U_m)\cap K)$. The set $C=\{ x_{n,m}: \langle n,m\rangle\in M\}$ is a countable subset of $K\setminus\{x\}$ such that $x\in\cl_{\mathbf{X}}(C)$. Hence, there is a sequence of points of $A$ which converges in $\mathbf{X}$ to the point $x$. This shows that $\mathbf{X}$ is Fr\'echet-Urysohn. 
\end{proof}

\begin{theorem}
\label{s6t3}
$[\mathbf{ZF}]$  The statement ``Every compact metrizable very $k$-space is sequential'' implies $\mathbf{CAC}_{fin}$.
\end{theorem}
\begin{proof}
Assuming that $\mathbf{CAC}_{fin}$ is false, we fix a family $\mathcal{A}=\{A_n: n\in\mathbb{N}\}$ of pairwise disjoint, non-empty finite sets  such that $\mathcal{A}$ does not have a partial choice function (see \cite[Form 10E]{hr}). We fix an element $\infty\notin\bigcup\mathcal{A}$. For every $n\in\mathbb{N}$, let $\rho_n$ be the discrete metric on $A_n$.  We put $X=\{\infty\}\cup\bigcup\mathcal{A}$, and consider the metric $d$ on $X$ defined by ($\ast$). In view of Proposition \ref{s6p1}, the metrizable space $\langle X, \tau(d)\rangle$ is a compact, very $k$-space which is not sequential.
\end{proof}

\begin{corollary}
\label{s6c4}
The statement ``For every metrizable space $\mathbf{X}$, $\mathbf{Hk}(\mathbf{X})$ and $\mathbf{FU}(\mathbf{X})$ are equivalent'' is unprovable in $\mathbf{ZF}$.
\end{corollary}
\begin{proof}
By Theorem \ref{s6t3}, the statement ``For every metrizable $\mathbf{X}$, $\mathbf{Hk}(\mathbf{X})$ and $\mathbf{FU}(\mathbf{X})$ are equivalent'' is false in every model of $\mathbf{ZF}+\neg\mathbf{CAC}_{fin}$, for instance, in Pincus' Model I (model $\mathcal{M}4$ in \cite{hr}). 
\end{proof}

\begin{theorem}
\label{s6t5}
$[\mathbf{ZF}]$ 
\begin{enumerate}
\item[(i)] For every second-countable metrizable space $\mathbf{X}$, the following holds:
\begin{enumerate} 
\item[(a)] $\mathbf{Hk}(\mathbf{X})\leftrightarrow \mathbf{FU}(\mathbf{X})$;
\item[(b)] $\mathbf{Hk}(\mathbf{X})\rightarrow\mathbf{UIC}(\mathbf{X})$.
\end{enumerate}
\item[(ii)] $\mathbf{CAC}(\mathbb{R})\leftrightarrow\mathbf{Hk}_{M2}\leftrightarrow \mathbf{k}_{M2}$.
\item[(iii)] $\mathbf{k}_{M2}\rightarrow\mathbf{UIC}_{M2}$.
\end{enumerate}
\end{theorem}
\begin{proof}
(i) Let $\mathbf{X}$ be a non-empty second-countable metrizable space.  By Theorem \ref{s1t13}, $\mathbf{X}$ is $\mathcal{K}$-Loeb, so (a) can be deduced from Theorem \ref{s6t2}. That (b) holds follows from (a) and Corollary \ref{s4c6}.\medskip

(ii) It is obvious that $\mathbf{Hk}_{M2}$ and $\mathbf{k}_{M2}$ are equivalent. It follows from (i) that $\mathbf{Hk}_{M2}$ implies $\mathbf{FU}(\mathbb{R})$. By Theorem \ref{s1t16}, $\mathbf{FU}(\mathbb{R})$, $\mathbf{CAC}(\mathbb{R})$ and $\mathbf{FU}_{M2}$ are equivalent. By Proposition \ref{s1p18}, $\mathbf{FU}_{M2}$ implies $\mathbf{Hk}_{M2}$.  All this taken together proves that (ii) holds. That (iii) is true follows from (i).
\end{proof}

\begin{remark} 
\label{s6r6}
$[\mathbf{ZF}]$ Let $\mathbf{X}$ be a Hausdorff space for which $\mathbf{UIC}(\mathbf{X})$ is false. Then we can fix an uncountable subset $P$ of $X$ such that every compact set of $\mathbf{X}$ which is contained in $P$  is finite. If, in addition $\mathbf{X}$ is second-countable, then  the subspace $\mathbf{P}$ of $\mathbf{X}$ is not  discrete, so,  there exists an accumulation point $p_0$ of $\mathbf{P}$. The set $A_0=P\setminus\{p_0\}$ is not closed in $\mathbf{P}$  but the intersection of $A_0$ with every compact subset of $\mathbf{P}$ is closed. This implies  that $\mathbf{P}$ is not a $k$-space. Hence, for every second-countable Hausdorff space $\mathbf{X}$ the following implication holds: $\mathbf{Hk}(\mathbf{X})\rightarrow \mathbf{UIC}(\mathbf{X})$.
\end{remark}

In \cite{kopsw}, it was shown  that $\mathbb{R}$ is a $k$-space in $\mathbf{ZF}$. Now, we can generalize this result by proving the following theorem:

\begin{theorem}
\label{s6t7}
$[\mathbf{ZF}]$ Every first-countable, $\sigma$-compact, $\mathcal{K}$-Loeb $T_3$-space is a $k$-space. In particular, every metrizable, second-countable $\sigma$-compact space is a $k$-space. 
\end{theorem}

\begin{proof}
Let $\mathbf{X}$ be a first-countable  $T_3$-space which is $\sigma$-compact and $\mathcal{K}$-Loeb. We fix a family $\{K_n: n\in\omega\}$ of compact subsets of $\mathbf{X}$ such that $X=\bigcup_{n\in\omega}K_n$.  Let $f$ be a choice function of $\mathcal{K}^{\ast}(\mathbf{X})$.  Let $A\subseteq X$ be a set such that, for every $K\in\mathcal{K}(\mathbf{X})$, the set $A\cap K\in\mathcal{C}l(\mathbf{X})$. To prove that $A\in\mathcal{C}l(\mathbf{X})$, we fix a point $x\in\cl_{\mathbf{X}}(A)$ and a base $\{U_i: i\in\omega\}$ of open neighborhoods of $x$ in $\mathbf{X}$ such that, for every $i\in\omega$, $\cl_{\mathbf{X}}(U_{i+1})\subseteq U_i$.  For every $i\in\omega$, let $n_i=\min\{n\in\omega: A\cap U_i\cap K_n\neq\emptyset\}$. For every $i\in\omega$, the set $A\cap\cl_{\mathbf{X}}(U_{i})\cap K_{n_i}$ is a non-empty compact set in $\mathbf{X}$, so we can define $x_i=f(A\cap\cl_{\mathbf{X}}(U_{i})\cap K_{n_i})$. The set $C=\{x\}\cup\{x_i: i\in\omega\}$ is compact in $\mathbf{X}$, so $C\cap A$ is closed in $\mathbf{X}$. Therefore, $x\in A$ because the sequence $(x_i)_{i\in\omega}$ of points of $C\cap A$  converges in $\mathbf{X}$ to the point $x$. This implies that $A=\cl_{\mathbf{X}}(A)$, so $A\in\mathcal{C}l(\mathbf{X})$. Hence $\mathbf{X}$ is a $k$-space.
To complete the proof, it suffices to recall that, by Theorem \ref{s1t13}, every second-countable metrizable space is $\mathcal{K}$-Loeb. 
\end{proof}

\begin{theorem}
\label{s6t8}
$[\mathbf{ZF}]$
\begin{enumerate}  
\item[(i)] The negation of $\mathbf{CAC}_{\omega}(\mathbb{R})$ implies that there exists a subspace $\mathbf{X}$ of $\mathbb{R}$ such that $\mathbf{X}$ is neither a $k$-space nor Loeb but $\mathbf{UIC}(\mathbf{X})$ holds.
\item[(ii)] If every subspace of $\mathbb{R}$ is either a $k$-space or Loeb, then $\mathbf{CAC}(\mathbb{R})$ holds.
\end{enumerate}
\end{theorem}
\begin{proof}
(i) Suppose that $\mathbf{CAC}_{\omega}(\mathbb{R})$ fails. Then, by \cite[Lemma 2(i)]{hkrst}, there exists a family $\mathcal{A}=\{A_n: n\in\mathbb{N}\}$ of countable dense subsets of $\mathbb{R}$ such that $\mathcal{A}$ does not have a partial choice function. For every $n\in\mathbb{N}$, we can effectively define in $\mathbf{ZF}$ a homeomorphism  $\psi_n:\mathbb{R}\to (\frac{1}{n+1}, \frac{1}{n})$. For every $n\in\mathbb{N}$, let $F_n=\psi_n[A_n]$. Let $X=\{0\}\cup\bigcup_{n\in\mathbb{N}}F_n$. For every $n\in\mathbb{N}$, the set $F_n$ is closed in the subspace $\mathbf{X}$ of $\mathbb{R}$.  Since $\mathcal{A}$ does not have a partial choice function, the family $\{F_n: n\in\mathbb{N}\}$ of closed sets of $\mathbf{X}$ does not have a choice function. This shows that $\mathbf{X}$ is not Loeb. The set $F=\bigcup_{n\in\mathbb{N}}F_n$ is not closed in $\mathbf{X}$ but $F$ is sequentially closed in $\mathbf{X}$. Let $K$ be a compact set in $\mathbf{X}$. Let $N_K=\{n\in\mathbb{N}: K\cap F_n\neq\emptyset\}$. Suppose that the set $N_K$  is infinite. Since, by Theorem \ref{s1t12}, $\mathbb{R}$ is Loeb, we can fix a Loeb function $f$ of $\mathbb{R}$. The sets $K\cap F_n$ are closed in $\mathbb{R}$. By assigning to every $n\in N_K$  the element $f(K\cap F_n)$ of $F_n$, we obtain a choice function of $\{F_n: n\in N_K\}$. This contradicts the assumption that $\mathcal{A}$ does not have a partial choice function. The contradiction obtained proves that the set $N_K$ is finite. Hence the set  $F\cap K=\bigcup_{n\in N_K}(K\cap F_n)$ is closed in $\mathbf{X}$. This proves that $\mathbf{X}$ is not a $k$-space.

To show that $\mathbf{UIC}(\mathbf{X})$ holds, let us fix an uncountable subset $U$ of $X$. Let $N_U=\{n\in\mathbb{N}: |U\cap F_n|=\aleph_0\}$. Since $\mathcal{A}$ does not have a partial choice function, the set $N_U$ is infinite. If, for every $n\in N_U$, the subspace $U\cap F_n$ of $\mathbf{X}$ were discrete, the set $U$ would be countable. Hence, there exists $n_0\in N_U$ such that $U\cap F_{n_0}$ is not discrete. Let $x_0\in U\cap F_{n_0}$ be an accumulation point of $U\cap F_{n_0}$. Since $U\cap F_{n_0}$ is denumerable, there exists a sequence $(x_n)_{n\in\mathbb{N}}$ of points of $(U\cap F_{n_0})\setminus\{x_0\}$ such that $\lim\limits_{n\to+\infty}x_n=x_0$ in $\mathbb{R}$. Then the set $E=\{x_n: n\in\omega\}\subseteq U$ is closed in $\mathbf{X}$ and infinite. Hence $\mathbf{UIC}(\mathbf{X})$ holds. \medskip

(ii) By \cite[Lemma 2(ii)]{hkrst}, assuming that $\mathbf{CAC}(\mathbb{R})$ is false, we can fix a family $\mathcal{A}=\{A_n: n\in\mathbb{N}\}$ of non-empty subsets of $\mathbb{R}$ such that $\mathcal{A}$ does not have a partial choice function. Arguing in much the same way, as in the proof of (i),  we fix a sequence $(\psi_n)_{n\in\mathbb{N}}$ of homeomorphisms  $\psi_n:\mathbb{R}\to (\frac{1}{n+1}, \frac{1}{n})$ and, for every $n\in\mathbb{N}$, we put $F_n=\psi_n[A_n]$. Then, for $X=\{0\}\cup\bigcup_{n\in\mathbb{N}}F_n$, the subspace $\mathbf{X}$ of $\mathbb{R}$ is neither a $k$-space nor Loeb.
\end{proof}

Item (i) of the following theorem generalizes \cite[Theorem 4.23 and Corollary 4.25]{kw1}:

\begin{theorem}
\label{s6t9}
$[\mathbf{ZF}]$
\begin{enumerate} 
\item[(i)] $\omega-\mathbf{CAC}(\mathbb{R})$  implies that every second-countable metrizable Loeb space is sequential, so also a $k$-space. In particular, $\omega-\mathbf{CAC}(\mathbb{R})$  implies that every second-countable Cantor completely metrizable space is sequential, so also a $k$-space.

\item[(ii)] The sentence "Every second-countable metrizable Loeb space is sequential" does not imply $\mathbf{CAC}_{\omega}(\mathbb{R})$. In consequence, the sentence "Every second-countable metrizable Loeb space is a $k$-space" does not imply $\mathbf{CAC}_{\omega}(\mathbb{R})$.

\item[(iii)] $\mathbf{CAC}(\mathbb{R})$ is equivalent to the sentence ``Every second-countable metrizable $k$-space is Fr\'echet-Urysohn''.

\item[(iv)] The sentence ``Every metrizable $k$-space is Fr\'echet-Urysohn'' implies both $\mathbf{CAC}(\mathbb{R})$ and $\mathbf{CAC}_{fin}$.
\end{enumerate}
\end{theorem}
\begin{proof}
(i) Let $\mathbf{X}$ be a second-countable metrizable Loeb space and let $A$ be a sequentially closed subset of $\mathbf{X}$.  To show that $A$ is closed in $\mathbf{X}$, we fix $x\in\cl_{\mathbf{X}}(A)$.  Let $d$ be any metric which induces the topology of $\mathbf{X}$ and let $f$ be a Loeb function of $\mathbf{X}$. Assuming $\omega-\mathbf{CAC}(\mathbb{R})$, we can fix a family $\{A_n: n\in\mathbb{N}\}$ of countable sets such that, for every $n\in\mathbb{N}$, $A_n\subseteq A\cap B_d(x, \frac{1}{n})$. Since $A$ is sequentially closed in $\mathbf{X}$, it follows that, for every $n\in\mathbb{N}$, $\cl_{\mathbf{X}}(A_n)\subseteq A$. For every $n\in\mathbb{N}$, we define $x_n=f(\cl_{\mathbf{X}}(A_n))$. Then $(x_n)_{n\in\mathbb{N}}$ is a sequence of points of $A$ which converges to $x$ in $\mathbf{X}$. Since $A$ is sequentially closed in $\mathbf{X}$, we deduce that $x\in A$. Hence $A$ is closed in $\mathbf{X}$; thus, $\mathbf{X}$ is sequential. That $\mathbf{X}$ is a $k$-space follows from Proposition \ref{s1p18}(ii). In view of Theorem \ref{s1t12}, the second statement of (i) follows from the first one.\medskip

(ii) In the Feferman-Levy Model $\mathcal{M}9$ in \cite{hr}, $\mathbb{R}$ is a countable union of countable sets, so $\omega-\mathbf{CAC}(\mathbb{R})$ is true in $\mathcal{M}9$. It is known that $\mathbf{CAC}(\mathbb{R})$ is false in $\mathcal{M}9$ (see, e.g., \cite[p. 77]{her}). Since the equivalence  
$$\mathbf{CAC}(\mathbb{R})\leftrightarrow((\omega-\mathbf{CAC}(\mathbb{R}))\wedge\mathbf{CAC}_{\omega}(\mathbb{R}))$$
 holds in $\mathbf{ZF}$, $\mathbf{CAC}_{\omega}(\mathbb{R})$ is false in $\mathcal{M}9$. It follows from (i) that it is true in $\mathcal{M}9$ that every second-countable metrizable Loeb space is sequential, so also a $k$-space by Proposition \ref{s1p18}(ii).\medskip
 
To prove (iii), it suffices to notice that, by Theorem \ref{s1t16}, $\mathbf{CAC}(\mathbb{R})$ is equivalent to both $\mathbf{FU}_{M2}$ and $\mathbf{FU}(\mathbb{R})$, and $\mathbb{R}$ is a second-countable $k$-space.
 
To prove (iv), in the light of Theorems \ref{s6t3} and \ref{s1t16}, it suffices to notice that if every metrizable $k$-space is Fr\'echet-Urysohn, then $\mathbf{FU}(\mathbb{R})$. 
 \end{proof}
 
 \begin{remark}
 \label{s6r10}
  Let us recall that, since the conjunction $(\omega-\mathbf{CAC}(\mathbb{R}))\wedge\neg\mathbf{CAC}(\mathbb{R})$ is true in the Feferman-Levy Model $\mathcal{M}9$ in \cite{hr}, $\mathbf{FU}(\mathbb{R})$ fails in $\mathcal{M}9$ (see \cite[Remark 4.59]{hr}). Hence, by Theorems \ref{s1t12} and  \ref{s6t9}(i), in $\mathbf{ZF}$,  the sentence ``Every second-countable Cantor completely metrizable space is sequential'' does not imply ``Every second-countable Cantor completely metrizable space is Fr\'echet-Urysohn''.
 \end{remark}
 
 One may ask the following question:
 \begin{question}
 \label{s6q11} 
  Are the sentences ``Every first-countable Hausdorff space is a $k$-space'' and ``Every first-countable Hausdorff space is Fr\'echet-Urysohn'' equivalent in $\mathbf{ZF}$?
 \end{question}
To give an answer to this question, first of all, let us prove the following theorem:

\begin{theorem}
\label{s6t12}
$[\mathbf{ZF}]$ Each of $\mathbf{CAC}$, $\mathbf{FU}_M$, $\mathbf{S}_M$, $\mathbf{HS}_M$ is equivalent to each of the following sentences (i)--(iii):
\item[(i)] every first-countable Hausdorff space is Fr\'echet-Urysohn;
\item[(ii)] every first-countable Hausdorff space is sequential;
\item[(iii)] every first-countable Hausdorff space is hereditarily sequential. 
\end{theorem}
\begin{proof}
It follows from Proposition \ref{s1p18}(i) that (i)--(iii) are all equivalent. By Proposition \ref{s1p18},  $\mathbf{FU}_M$, $\mathbf{S}_M$ and $\mathbf{HS}_M$ are also all equivalent. Of course, $\mathbf{CAC}$ implies (i)--(iii) and each of $\mathbf{FU}_M$, $\mathbf{S}_M$ and $\mathbf{HS}_M$. 

To complete the proof, suppose that $\mathbf{CAC}$ is false. Then we can fix a family $\mathcal{A}=\{A_n: n\in\omega\}$ of pairwise disjoint non-empty sets such that $\mathcal{A}$ does not have a partial choice function (see \cite[Form 8B]{hr}). It follows from Proposition \ref{s6p1} that there exists a metrizable space which is not sequential. Hence $\mathbf{S}_M$ and each of (i)--(iii) implies $\mathbf{CAC}$.
\end{proof}

\begin{remark}
\label{s6r13}
To show that, in $\mathbf{ZFA}$, the answer to Question \ref{s6q11} is in the negative, we recall that $\mathbf{CMC}$ does not imply $\mathbf{CAC}$ in $\mathbf{ZFA}$. Indeed, for instance, in the Second Fraenkel Model $\mathcal{N}2$ in \cite{hr} and in Brunner's Model III ( $\mathcal{N}50(E)$ in \cite{hr}), $\mathbf{CMC}\wedge\neg\mathbf{CAC}$ is true (see \cite[p. 178]{hr} and \cite[p. 219]{hr}). Moreover, Theorems \ref{s1:fil} and \ref{s6t12} are valid in $\mathbf{ZFA}$. Hence, in $\mathcal{N}2$ and  $\mathcal{N}50(E)$,  the sentences ``Every first-countable Hausdorff space is a $k$-space'' and ``Every first-countable Hausdorff space is Fr\'echet-Urysohn'' are not equivalent. Unfortunately, we do not know if $\mathbf{CMC}$ implies $\mathbf{CAC}$ in $\mathbf{ZF}$. 
\end{remark}

Now, let us turn our attention to the forms $\mathbf{WUF}(\mathbf{X})$ and $\mathbf{CUF}(\mathbf{X})$.

\begin{proposition}
\label{s6p14}
$[\mathbf{ZF}]$ 
\begin{enumerate}
\item[(i)]For every sequential space $\mathbf{X}$, $\mathbf{WUF}(\mathbf{X})$ holds.

\item[(ii)] For every second-countable metrizable Loeb space, $\mathbf{WUF}(\mathbf{X})$ holds. 

\item[(iii)] For every Hausdorff space $\mathbf{X}$, $\mathbf{WUF}(\mathbf{X})$ implies $\mathbf{CUF}(\mathbf{X})$. 

\item[(iv)] For every $k$-space $\mathbf{X}$, $\mathbf{WUF}(\mathbf{X})$ and $\mathbf{CUF}(\mathbf{X})$ are equivalent. 
\end{enumerate}
\end{proposition}
\begin{proof} 
(i) Suppose that $\mathbf{X}$ is a sequential space, $F\in\mathcal{C}l(\mathbf{X})$ and $x$ is an accumulation point of $F$ in $\mathbf{X}$. Since $F\setminus\{x\}$ is not closed in the sequential space $\mathbf{X}$, $F\setminus\{x\}$ is not sequentially closed in $\mathbf{X}$. Hence, there exists a sequence $(x_n)_{n\in\omega}$ of points of $F\setminus\{x\}$ which converges in $\mathbf{X}$ to a point $y$ such that $y\not\in F\setminus\{x\}$. Since $F$ is closed in $\mathbf{X}$, $y\in F$. Therefore, $y=x$. This proves that $\mathbf{WUF}(\mathbf{X})$ is satisfied.  \medskip

(ii) Now, suppose that $\mathbf{X}$ is a second-countable metrizable Loeb space. Then, by Theorem \ref{s2t3}, every closed subspace of $\mathbf{X}$ is separable, so $\mathbf{WUF}(\mathbf{X})$ holds.\medskip

It is obvious that (iii) holds. To prove (iv), we assume that $\mathbf{X}$ is a $k$-space for which $\mathbf{CUF}(\mathbf{X})$ holds. Suppose that $A\in\mathcal{C}l(\mathbf{X})$, and $a$ is an accumulation point of $A$. Since $A\setminus\{a\}$ is not closed in the $k$-space $\mathbf{X}$, there exists $K\in\mathcal{K}(\mathbf{X})$ such that $K\cap(A\setminus\{a\})\notin\mathcal{C}l(\mathbf{X})$. Since $K\cap A\in\mathcal{C}l(\mathbf{X})$, $a$ is an accumulation point of $K\cap A$ and $a\in K\cap A$. By $\mathbf{CUF}(\mathbf{X})$, there exists a sequence of points of $(K\cap A)\setminus\{a\}$ which converges in $\mathbf{X}$ to $a$. Hence $\mathbf{WUF}(\mathbf{X})$ holds. To complete the proof of (iv), it suffices to apply (iii).
\end{proof}

\begin{theorem}
\label{s6t15}
$[\mathbf{ZF}]$
\begin{enumerate}
\item[(i)] If, for every second-countable metrizable space $\mathbf{X}$, $\mathbf{WUF}(\mathbf{X})$ implies $\mathbf{S}(\mathbf{X})$, then both $\mathbf{S}(\mathbb{R})$ and  $\mathbf{IDI}(\mathbb{R})$ hold.

\item[(ii)] $\mathbf{WUF}_M\leftrightarrow \mathbf{CAC}$.

\item[(iii)] $\mathbf{WUF}_M\rightarrow\mathbf{CUF}_M\rightarrow \mathbf{CAC}_{fin}$.

\item[(iv)] $\mathbf{CUF}_M$ implies that the following holds: For every family $\{\langle A_n, d_n\rangle: n\in\mathbb{N}\}$ of non-empty compact metric spaces, the family $\{A_n: n\in\mathbb{N}\}$ has a choice function.

\item[(v)] $\mathbf{CAC}_{fin}\nrightarrow \mathbf{CUF}_M$.
\end{enumerate}
\end{theorem}
\begin{proof}
(i) Suppose that, for every second-countable metrizable space $\mathbf{X}$, $\mathbf{WUF}(\mathbf{X})$ implies $\mathbf{S}(\mathbf{X})$. In view of Theorem \ref{s1t12} and Proposition \ref{s6p14}(ii), $\mathbf{WUF}(\mathbb{R})$ holds in $\mathbf{ZF}$. Thus, since $\mathbb{R}$ is also second-countable and metrizable, it follows from our assumption that $\mathbf{S}(\mathbb{R})$ holds. It is known that $\mathbf{S}(\mathbb{R})$ implies $\mathbf{IDI}(\mathbb{R})$ because every infinite Dedekind-finite subset of $\mathbb{R}$ is sequentially closed but not closed in $\mathbb{R}$. Hence (i) holds.

(ii)--(iii) It is obvious that the implications $\mathbf{CAC}\rightarrow\mathbf{WUF}_M\rightarrow\mathbf{CUF}_M$ are true. Suppose that $\mathbf{CAC}$ is false. Let $\mathcal{A}=\{A_n: n\in\mathbb{N}\}$ be a pairwise disjoint family of non-empty sets such that $\mathcal{A}$ does not have a partial choice function. For every $n\in\mathbb{N}$, let $\rho_n$ be the discrete metric on $A_n$. We fix an element $\infty\notin\bigcup\mathcal{A}$ and put $X=\{\infty\}\cup\bigcup\mathcal{A}$. Let $d$ be the metric on $\mathbf{X}$ defined by $(\ast)$. Then $\infty$ is an accumulation point of $X$ in the space $\mathbf{X}=\langle X, \tau(d)\rangle$ but, since $\mathcal{A}$ does not have a partial choice function, no sequence of points of $\bigcup\mathcal{A}$ converges in $\mathbf{X}$ to $\infty$. Hence $\mathbf{WUF}(\mathbf{X})$ is false. We observe that if, for every $n\in\mathbb{N}$, the set $A_n$ is finite, then, by Proposition \ref{s6p1}, the space $\mathbf{X}$ is compact, so $\mathbf{CUF}(\mathbf{X})$ is false. In consequence, $\mathbf{WUF}_M$ implies $\mathbf{CAC}$ and $\mathbf{CUF}_M$ implies $\mathbf{CAC}_{fin}$. Hence (ii) and (iii) hold.\medskip

(iv) Now, assume $\mathbf{CUF}_M$ and consider any family $\{\langle A_n, \rho_n\rangle: n\in\mathbb{N}\}$ of compact metric spaces such that, for every pair $m,n$ of distinct elements of $\mathbb{N}$, $A_m\cap A_n=\emptyset$. Let $\mathcal{A}=\{A_n: n\in\mathbb{N}\}$ and let $\infty$ be an element such that $\infty\notin\bigcup\mathcal{A}$. As in the proof of (ii)--(iii), we put $X=\{\infty\}\cup\bigcup\mathcal{A}$ and consider the metric $d$ on $X$ defined by $(\ast)$. By Proposition \ref{s6p1}, the metric space $\langle  X, d\rangle$ is compact. Since $\infty$ is an accumulation point of $X$ in $\mathbf{X}=\langle X, \tau(d)\rangle$, it follows from $\mathbf{CUF}(\mathbf{X})$ that there exists a sequence of points of $\bigcup\mathcal{A}$ which converges in $\mathbf{X}$ to the point $\infty$. This implies that $\mathcal{A}$ has a partial choice function. For every $n\in\mathbb{N}$, let $Y_n=\prod_{i=1}^n A_i$ and let $\rho_{Y_n}$ be the metric on $Y_n$ defined by: $\rho_{Y_n}(y,z)=\max\{\rho_i(y(i), z(i)): i\in\{1,\dots n\}\}$ for $y,z\in Y_n$. Then, for every $n\in\mathbb{N}$, the metric space $\langle Y_n, \rho_{Y_n}\rangle$ is compact. It follows from  the first step of the proof of (iv) that the family $\{Y_n: n\in \mathbb{N}\}$ has a partial choice function. This implies that $\mathcal{A}$ has a choice function.\medskip

(v) By \cite[Proposition 10]{ktw3}, there is a model $\mathcal{M}$ of $\mathbf{ZF}+\mathbf{CAC}_{fin}$ in which there exists a compact metrizable space which is not separable.  It is known from \cite{kert1} that the sentence ``Every compact metric space is separable'' is equvalent in $\mathbf{ZF}$ to the sentence ``For every family $\{\langle A_n, d_n\rangle: n\in\mathbb{N}\}$ of non-empty compact metric spaces, the family $\{A_n: n\in\mathbb{N}\}$ has a choice function''. Hence, by (iv), $\mathbf{CUF}_M$ fails in $\mathcal{M}$. 
\end{proof}

\section{What to add to $\mathbf{Hk}(\mathbf{X})$ to get $\mathbf{FU}(\mathbf{X})$?}
\label{s7}

Let us show that, for every first-countable $T_3$-space, $\mathbf{CPCAC}(\mathbf{X})$ is the missing portion of $\mathbf{AC}$ to get $\mathbf{FU}(\mathbf{X})$ equivalent to $\mathbf{Hk}(\mathbf{X})$. 

\begin{theorem}
\label{s7t1}
$[\mathbf{ZF}]$ Let $\mathbf{X}=\langle X, \tau\rangle$ be a first-countable Hausdorff space. Then:
\begin{enumerate}
\item[(i)]  $\mathbf{FU}(\mathbf{X})\leftrightarrow (\mathbf{CPCAC}(\mathbf{X})\wedge\mathbf{Hk}(\mathbf{X}))$;

\item[(ii)] $\mathbf{CPCAC}(\mathbf{X})\nrightarrow \mathbf{k}(\mathbf{X})$ and $%
\mathbf{Hk}(\mathbf{X})\nrightarrow\mathbf{CPCAC}(\mathbf{X})$.
\end{enumerate}
\end{theorem}
\begin{proof}
(i) ($\rightarrow$) That $\mathbf{FU}(\mathbf{X})$ implies $\mathbf{Hk}(%
\mathbf{X})$ follows from Proposition \ref{s1p18}. To show that $\mathbf{FU}(\mathbf{X})$ implies  $\mathbf{CPCAC}(\mathbf{X})$, we assume that $\mathbf{X}$ is Fr\'echet-Urysohn. We fix an
infinite compact subset $K$ of $\mathbf{X}$ and a family $\mathcal{A}%
=\{A_{i}: i\in \omega\}$ of non-empty subsets of $K$. Without
loss of generality, we may assume that, for every infinite subset $M$ of $\omega$,  $\bigcap_{i\in M}A_i=\emptyset$. Let $\mathcal{U}=\{U\in\tau: \{i\in\omega: U\cap A_i\neq\emptyset\}\in[\omega]^{<\omega}\}$. By the compactness of $K$ in $\mathbf{X}$, $K\setminus\bigcup\mathcal{U}\neq\emptyset$. Thus, we can fix $x\in K\setminus\bigcup\mathcal{U}$. Let $\{U_n: n\in\omega\}$ be a base of neighborhoods of $x$ in $\mathbf{X}$ such that, for every $n\in\omega$, $U_{n+1}\subseteq U_n$. For every $n\in\omega$, the set $M_n=\{i\in\omega: U_n\cap (A_i\setminus\{x\})\neq\emptyset\}$ is infinite. Therefore, we can inductively define strictly increasing sequences $%
(k_{i})_{i\in\omega},(m_{i})_{i\in \omega}$ of natural numbers such that, for every $i\in \omega$, the set $(U_i\setminus U_{m_i})\cap (A_{k_{i}}\setminus \{x\})$ is non-empty.  Let $Y=\bigcup \{A_{k_{i}}:i\in\omega\}\setminus\{x\}$.
Since $x\in \cl_{\mathbf{X}}(Y)$ and $\mathbf{X}$ is Fr\'echet-Urysohn, we can fix a sequence $(x_{n})_{n\in \omega}$
of points of $Y$ which converges in $\mathbf{X}$ to the point $x$. With  $(x_{n})_{n\in \mathbb{N}%
}$ in hand,  we can easily define  an infinite subset $M$ of $\omega$ such that $\prod_{i\in M}A_i\neq\emptyset$. Hence $\mathbf{FU}(\mathbf{X})$ implies $\mathbf{CPCAC}(\mathbf{X})$. 

($\leftarrow $) Now, we assume $\mathbf{CPCAC}(\mathbf{X})\wedge\mathbf{Hk}(\mathbf{X})$. Let $C\subseteq X$ and $C\notin\mathcal{C}l(\mathbf{X})$. We fix $x\in\cl_{\mathbf{X}}(C)\setminus  C$.  Put $P=C\cup\{x\}$. The subspace $\mathbf{P}$ of $\mathbf{X}$ is a $k$-space, while $C$ is not closed in $\mathbf{P}$. Hence, there exists $F\in\mathcal{K}(\mathbf{P})$ such that $F\cap C\notin\mathcal{C}l(\mathbf{P})$. Then $x\in\cl_{\mathbf{X}}(F\cap C)$. As before, let $\{U_n: n\in\omega\}$ be a base of neighborhoods of $x$ in $\mathbf{X}$ such that, for every $n\in\omega$, $U_{n+1}\subseteq U_n$. For every $n\in \omega$, let $C_{n}=U_n\cap (F\cap C)$. Then $\{C_n: n\in\omega\}$ is a family of non-empty subsets of the compact set $F$. By $\mathbf{CPCAC}(\mathbf{X})$, there exists an infinite set $M\subseteq\omega$ such that $\prod_{n\in M}C_n\neq\emptyset$. If $\psi\in\prod_{n\in M}C_n$, then $D=\{\psi(n): n\in M\}$ is a denumerable subset of $C$ such that $x\in\cl_{\mathbf{X}}(D)$. This implies that there exists a sequence of points of $D$ (so also of $C$) which converges in $\mathbf{X}$ to $x$. Hence $\mathbf{FU}(\mathbf{X})$ holds.\medskip

(ii) It follows from the proof of Theorem \ref{s6t3} that, in every model of $\mathbf{ZF}+\neg\mathbf{CAC}_{fin}$, there exists a compact metrizable space $\mathbf{X}$ for which the conjunction $\mathbf{Hk}(\mathbf{X})\wedge\neg\mathbf{CPCAC}(\mathbf{X})$ is true. 

In Cohen's Original Model $\mathcal{M}1$ of \cite{hr}, the set $A$ of all added
Cohen reals is an infinite Dedekind-finite subset of $\mathbb{R}$. Since, in $\mathcal{M}1$,  every compact subset of the subspace $\mathbf{A}$ of $\mathbb{R}$ is finite, $\mathbf{CPCAC}(\mathbf{A})$ is true but $\mathbf{A}$ is not a $k$-space.
\end{proof}

\begin{corollary}
\label{s7c2}
$[\mathbf{ZF}]$ For every first-countable, $\mathcal{K}$-Loeb $T_3$-space $\mathbf{X}$, $\mathbf{Hk}(\mathbf{X})$ implies $\mathbf{CPCAC}(\mathbf{X})$.
\end{corollary}
\begin{proof}
This follows directly from Theorems \ref{s6t2} and \ref{s7t1}.
\end{proof}

The following proposition is obvious:

\begin{proposition}
\label{s7p3}
$[\mathbf{ZF}]$
\begin{enumerate}
\item[(i)] For every topological space $\mathbf{X}$, the following is true:
 $$(\mathbf{PCAC}(\mathbf{X})\rightarrow\mathbf{CPCAC}(\mathbf{X}))\wedge (\mathbf{PCMC}(\mathbf{X})\rightarrow\mathbf{CPCMC}(\mathbf{X})).$$
\item[(ii)] For every compact space $\mathbf{X}$, the following is true:
 $$(\mathbf{PCAC}(\mathbf{X})\leftrightarrow\mathbf{CPCAC}(\mathbf{X}))\wedge (\mathbf{PCMC}(\mathbf{X})\leftrightarrow\mathbf{CPCMC}(\mathbf{X})).$$
 \item[(iii)] For every $\mathcal{K}$-Loeb $T_1$-space $\mathbf{X}$, the following is true:
 $$(\mathbf{PCAC}(\mathbf{X})\leftrightarrow\mathbf{PCMC}(\mathbf{X}))\wedge(\mathbf{CPCAC}(\mathbf{X})\leftrightarrow\mathbf{CPMC}(\mathbf{X})).$$
\end{enumerate}
\end{proposition}

\begin{proposition}
\label{s7p4}
$[\mathbf{ZF}]$ Let $\mathbf{Y}=\langle Y, \tau\rangle$ be a Loeb $T_3$-space which has a cuf base. If $\mathbf{IDI}(Y)$ is false, then there exists a subspace $\mathbf{X}$ of $\mathbf{Y}$ such that $\mathbf{CPCAC}(\mathbf{X})$ is true but $\mathbf{PCMC}(\mathbf{X})$ is false. 
\end{proposition}
\begin{proof}
Suppose that $X$ is an infinite Dedekind-finite subset of $Y$. It follows from Theorem \ref{s2t5} that every compact subset of the subspace $\mathbf{X}$ of $\mathbf{Y}$ is finite, so $\mathbf{CPCAC}(\mathbf{X})$ holds. By Theorem \ref{s2t5}, the set $X$ is not closed in $\mathbf{Y}$. Let us fix $x_0\in\cl_{\mathbf{Y}}(X)\setminus X$. By Theorem \ref{s2t3}, the space $\mathbf{Y}$ is metrizable and even second-countable. Let $\{U_n: n\in\omega\}$ be a base of neighborhoods of $x_0$ in $\mathbf{Y}$ such that, for every $n\in\omega$, $U_{n+1}\subseteq U_n$. For every $n\in\omega$, let $A_n=X\cap U_n$. Suppose that the family $\{A_n: n\in\omega\}$ has a partial multiple choice function. Let $M$ be an infinite subset of $\omega$ such that the family $\{A_n: n\in M\}$ has a multiple choice function $h$. Let $f$ be a Loeb function of $\mathbf{Y}$. Then the set $\{f(h(n)): n\in M\}$ is a denumerable subset of $X$ which cotradicts the assumption that $X$ is Dedekind-finite. Hence $\mathbf{PCMC}(\mathbf{X})$ is false. 
\end{proof}

\begin{remark}
\label{s7r5}
In Cohen's Original Model $\mathcal{M}1$ in \cite{hr}, if $A$ stands for the set of all added Cohen reals, then $A$ is an infinite Dedekind-finite subset of $\mathbb{R}$, so, by Proposition \ref{s7p4}, for the subspace $\mathbf{A}$ of $\mathbb{R}$, $\mathbf{CPCAC}(\mathbf{A})$ is true but $\mathbf{PCMC}(\mathbf{A})$ is false.
\end{remark}

\begin{remark} 
\label{s7r6}
Since in $\mathbf{ZFC}$, a $k$-space need not be Fr\'echet-Urysohn, for a Hausdorff space $\mathbf{X}$, the following implication may be false in $\mathbf{ZF}$: 
$$\text{ (1) } (\mathbf{CPCAC}(\mathbf{X})\wedge\mathbf{k}(\mathbf{X}))\rightarrow \mathbf{FU}(\mathbf{X}).$$
 However, we do not know if, for every Hausdorff first-countable space $\mathbf{X}$, (1) is true in $\mathbf{ZF}$. We also do not know if, for every Hausdorff first-countable space $\mathbf{X}$, the following implication is true in $\mathbf{ZF}$:
 $$\text{ (2) } (\mathbf{CPCAC}(\mathbf{X})\wedge\mathbf{S}(\mathbf{X}))\rightarrow \mathbf{FU}(\mathbf{X}).$$
\end{remark}

\begin{theorem}
\label{s7t7}
$[\mathbf{ZF}]$ Let $\mathbf{X}$ be a first-countable Hausdorff space such that, for every $A\subseteq X$ and every $x\in\cl_{\mathbf{X}}(A)$, there exists $K\in\mathcal{K}(\mathbf{X})$ such that $x\in\cl_{\mathbf{X}}(A\cap K)$. Then $\mathbf{FU}(\mathbf{X})$ and $\mathbf{CPCAC}(\mathbf{X})$ are equivalent. If $\mathbf{X}$ is also a $\mathcal{K}$-Loeb regular space, then $\mathbf{Hk}(\mathbf{X})$, $\mathbf{FU}(\mathbf{X})$ and $\mathbf{CPCAC}(\mathbf{X})$ are all equivalent.
\end{theorem}

\begin{proof} Let us assume $\mathbf{CPCAC}(\mathbf{X})$. To prove $\mathbf{FU}(\mathbf{X})$, suppose that $A\subseteq X$ and $x\in\cl_{\mathbf{X}}(A)\setminus A$. We fix $K\in\mathcal{K}(\mathbf{X})$ such that $x\in\cl_{\mathbf{X}}(A\cap K)$. Let $\{U_n: n\in\omega\}$ be a base of neighborhoods of $x$ in $\mathbf{X}$ such that, for every $n\in\omega$, $U_{n+1}\subseteq U_n$. Since the family $\mathcal{A}=\{U_n\cap A\cap K: n\in\omega\}$ has a partial choice function, there exists a denumerable subset $D$ of $A$ such that $x\in\cl_{\mathbf{X}}(D)$. Hence,  there exists a sequence of points of $A$ which converges in $\mathbf{X}$ to the point $x$. Therefore,  $\mathbf{CPCAC}(\mathbf{X})$ implies $\mathbf{FU}(\mathbf{X})$. Thus, by Theorem \ref{s7t1}(i), $\mathbf{CPCAC}(\mathbf{X})$ and $\mathbf{FU}(\mathbf{X})$ are equivalent. To conclude the proof, it suffices to apply Theorem \ref{s6t2}.
\end{proof}

\begin{corollary}
$[\mathbf{ZF}]$. Let $\mathbf{X}$ be a locally compact, first-countable Hausdorff space. Then $\mathbf{FU}(\mathbf{X})$ and $\mathbf{CPCAC}(\mathbf{X})$ are equivalent. If $\mathbf{X}$ is also  $\mathcal{K}$-Loeb, then $\mathbf{Hk}(\mathbf{X})$, $\mathbf{FU}(\mathbf{X})$ and $\mathbf{CPCAC}(\mathbf{X})$ are all equivalent.
\end{corollary}

\section{The shortlist of open problems}
\label{s8}

\begin{enumerate}
\item Is there a model of $\mathbf{ZF}$ or $\mathbf{ZFA}$ in which there exists a Cantor completely metrizable $\mathcal{K}$-Loeb space $\mathbf{X}$ such that $\mathbf{X}$ has a cuf base but is not Loeb? (Cf. Question \ref{s2q12}.)

\item  Does $\mathbf{IDI}(\mathbb{R})$ imply $\mathbf{UISC}(\mathbb{R})$ in $\mathbf{ZF}$? (Cf. Question \ref{s4q8}.)

\item Is there a model $\mathcal{M}$ of $\mathbf{ZF}+\mathbf{IDI}$ in which there exists a second-countable Cantor completely metrizable space $\mathbf{X}$ for which $\mathbf{UISC}(\mathbf{X})$ fails in $\mathcal{M}$? (Cf. Question \ref{s4q8}.)

\item Is it provable in $\mathbf{ZF}$ that if $\mathbf{X}$ is a first-countable Hausdorff space, then the conjunction $\mathbf{CPCAC}(\mathbf{X})\wedge\mathbf{k}(\mathbf{X})$ implies $\mathbf{FU}(\mathbf{X})$? (Cf. Remark \ref{s7r6}.)

\item Is it provable in $\mathbf{ZF}$ that if $\mathbf{X}$ is a first-countable Hausdorff space, then the conjunction $\mathbf{CPCAC}(\mathbf{X})\wedge\mathbf{S}(\mathbf{X})$ implies $\mathbf{FU}(\mathbf{X})$? (Cf. Remark \ref{s7r6}.)

\item  Does the sentence ``Every first-countable Hausdorff space is a $k$-space'' imply ``Every first-countable Hausdorff space is Fr\'echet-Urysohn'' in $\mathbf{ZF}$? (Cf. Question \ref{s6q11} and Remark \ref{s6r13}.)

\item Is the space $\mathbb{P}$ of irrationals a $k$-space in $\mathbf{ZF}$?

\item Is the implication $\mathbf{S}(\mathbb{R})\rightarrow\mathbf{S}(\mathbb{P})$ true in $\mathbf{ZF}$? (Cf. \cite[Problem 4.18]{kw1}.)

\end{enumerate}

\end{document}